\documentclass[12pt]{amsart}
\usepackage[dvips]{color}
\usepackage{amsmath}
\usepackage{amsxtra}
\usepackage{amscd}
\usepackage{amsthm}
\usepackage{amsfonts}
\usepackage{amssymb}
\usepackage{eucal}
\usepackage{epsfig}
\usepackage{graphics}
\def\({\left(}
\def\){\right)}

\newcommand{\gl}{\mathfrak{gl}}
\newcommand{\sln}{\mathfrak{sl}}
\newcommand{\ga}{\gamma}
\newcommand\Ref{\eqref}

\newcommand{\bra}[1]{\langle #1 |}        
\newcommand{\ket}[1]{{| #1 \rangle}}      

\newcommand{\bea}{\begin{eqnarray}}
\newcommand{\ena}{\end{eqnarray}}
\def\bel{\begin{eqnarray}}
\def\enl{\end{eqnarray}}
\newcommand{\be}{\begin{eqnarray*}}
\newcommand{\en}{\end{eqnarray*}}

\newcommand{\C}{{\mathbb C}}
\newcommand{\Z}{{\mathbb Z}}
\newcommand{\Q}{{\mathbb Q}}

\newenvironment{tenumerate}{
  \begin{enumerate}
  
  }{\end{enumerate}}
\newcommand{\bi}{\begin{tenumerate}}
\newcommand{\ei}{\end{tenumerate}}
\newcommand{\isoto}[1][]%
{{\mathop{\buildrel{\sim}\over\longrightarrow}\limits_{#1}}}

\def\[{\left[}
\def\]{\right]}
\newcommand{\la}{\lambda}
\newcommand{\La}{\Lambda}

\newcommand{\al}{\alpha}

\newcommand{\s}{\sigma}

\newcommand{\bs}{\boldsymbol}

\newcommand{\M}{{\mathcal M}}
\newcommand{\N}{{\mathcal N}}
\newcommand{\W}{{\mathcal W}}

\numberwithin{equation}{section}
\newtheorem{thm}{Theorem}[section]
\newtheorem{prop}[thm]{Proposition}
\newtheorem{lem}[thm]{Lemma}

\newtheorem{conj}[thm]{Conjecture}

\newcommand{\E}{{\mathcal E}}

\newcommand{\on}{\operatorname}

\newcommand{\F}{\mathcal F}
\newcommand{\Hc}{\mathcal H}

\def\bi{\mathbf{i}}

\begin{document}

\begin{title}[Representations of quantum toroidal  $\mathfrak{gl}_n$.]
{Representations of quantum toroidal  $\mathfrak{gl}_n$}
\end{title}
\author{B. Feigin, M. Jimbo, T. Miwa and E. Mukhin}
\address{BF: Landau Institute for Theoretical Physics,
Russia, Chernogolovka, 142432, prosp. Akademika Semenova, 1a,   \newline
Higher School of Economics, Russia, Moscow, 101000,  Myasnitskaya ul., 20 and
\newline
Independent University of Moscow, Russia, Moscow, 119002,
Bol'shoi Vlas'evski per., 11}
\email{bfeigin@gmail.com}
\address{MJ: Department of Mathematics,
Rikkyo University, Toshima-ku, Tokyo 171-8501, Japan}
\email{jimbomm@rikkyo.ac.jp}
\address{TM: Department of Mathematics,
Graduate School of Science,
Kyoto University, Kyoto 606-8502,
Japan}\email{tmiwa@kje.biglobe.ne.jp}
\address{EM: Department of Mathematics,
Indiana University-Purdue University-Indianapolis,
402 N.Blackford St., LD 270,
Indianapolis, IN 46202, USA}\email{mukhin@math.iupui.edu}

\begin{abstract} We define and study representations of quantum toroidal $\gl_n$ with natural bases labeled by plane partitions with various conditions. As an application, 
we give an explicit description of a family of highest weight representations of quantum affine 
$\gl_n$ with generic level.
\end{abstract}
\maketitle
\section{Introduction}
Quantum toroidal algebras, generalizations of quantum affine algebras, were introduced in \cite{GKV} .
The affine Lie algebra of type $\mathfrak g$ is the Lie algebra of currents $\C^\times\rightarrow\mathfrak g$,
where $\mathfrak g$ is a finite dimensional Lie algebra. The toroidal algebra of type $\mathfrak g$ is the Lie algebra of 
currents $\C^\times\times\C^\times\rightarrow\mathfrak g$. The quantum version of the toroidal algebra is written 
in terms of Drinfeld generators 
similarly to the quantum affine algebra changing the Cartan matrix to the affine version.

In the case of $\mathfrak g=\mathfrak{gl}_n$, the quantum toroidal algebra has an alternative construction.
For $q_1\in\C^\times$, let $A_n(q_1)=M_n\otimes\C[Z^{\pm1},D^{\pm1}]$, where
$M_n$ is the algebra of $n\times n$ matrices and $\C[Z^{\pm 1},D^{\pm1}]$ is the algebra of functions on quantum torus $ZD=q_1DZ$. 
The algebra $A_n(q_1)$ is isomorphic to the algebra of difference operators of the form
$\sum_{i\in\Z}f_i(Z)T^i$, where $f_i(Z)$ a matrix valued Laurent polynomial, and $T$ is the shift operator $Z\rightarrow q_1Z$.
Let $L_n(q_1)$ be the Lie algebra associated with $A_n(q_1)$. One can see that $L_n(q_1)$
has an invariant quadratic form and has a triangular decomposition.
Quantizing $UL_n(q_1)$ one obtains an algebra depending on two parameters
$q_1,q_2$, where $q_2$ is a parameter of quantization. After adding a two dimensional central extension, 
one arrives at the quantum toroidal algebra which we denote $\E^c_n$.

\medskip

It has been long believed that the representation theory of the quantum toroidal algebras is a very technical and difficult subject.
While this is definitely true, we argue that the category of representations of $\E^c_n$,
where one of the central elements, $q^c$, acts by one and the modules have finite-dimensional weight spaces is rich, beautiful and relatively easy to handle.
We denote $\mathcal E_n$ the algebra which is a
quotient of $\E_n^c$ by the relation $q^c=1$. 
In this paper 
we study $\E_n$ lowest weight modules with rational lowest weights.
Following \cite{M}, we call such representations quasi-finite. A quasi-finite representation has level given by the value of the remaining central element.

\medskip
The algebra $\mathcal E_n$ depends on two parameters $q_1,q_2$ but it is convenient to use three parameters $q_1,q_2,q_3$ with the property $q_1q_2q_3=1$. The parameters $q_1$ and $q_3$ can be swapped by an automorphism but $q_2$ is distinguished.
 
 The algebra $\mathcal E_n$ has a vector representation, $V(u)$, where $u\in\C^\times$ is an evaluation parameter. The vector representation is the quantum version of $L_n(q_1)$-module $\C^n\otimes\C[Z^{\pm1}]$. This representation has trivial level, and it 
 is our starting point for the construction of quasi-finite representations.
 First, we use a semi-infinite wedge construction with vector representations to 
 construct a family of Fock modules $\mathcal F^{(k)}(u)$ of level $q$, where $q^2=q_2$.
 Here $u\in\C^\times$ is an evaluation parameter, and 
$k\in\{0,1,\ldots,n-1\}$ is an additional parameter which we call ``color". The latter appears because the semi-infinite construction depends on the choice of one weight 
 vector in $\C^n$. 
 A basis in $\mathcal F^{(k)}(u)$ is labeled by  colored partitions, see Figure 2. 
 Next, using the semi-infinite wedge construction with Fock modules, we construct a family  of 
representations 
$\M^{(k)}_{\al,\beta,\ga}(u;K)$ depending on $k\in{\{0,1,\ldots,n-1\}}$, complex parameters $u,K$ and partitions
$\al,\beta,\ga$.
The basis of $\M^{(k)}_{\al,\beta,\ga}(u;K)$ 
is labeled by colored plane partitions with asymptotic boundary conditions along the axes, see Figure 3.
We call $\M^{(k)}_{\al,\beta,\ga}(u;K)$ the Macmahon module.
Note that for our semi-infinite construction to work, 
the partition $\ga$ has to have equal number of boxes of each color. We call such partitions colorless.
Finally, using tensor products of Macmahon modules and specializing their levels, we obtain a large family of irreducible, quasi-finite $\E_n$-representations with bases labeled by tuples of plane partitions with various boundary conditions.

\medskip

Our constructions have various applications to the other areas of mathematics. 
The quantum toroidal algebra has a subalgebra isomorphic to $U_q\widehat{\gl}_n$, 
given in Drinfeld-Jimbo form.
As one application, we consider a certain class of $\E_n$-modules $\mathcal G_{\mu,\nu}^{(k)}(u)$ of level $q_1^{n/2}$
which stay irreducible as 
$U_q\widehat{\gl}_n$-modules. 
We call a $U_q\widehat{\gl}_n$-module of generic level Weyl type if it possesses a BGG type resolution by Verma modules of the same form as finite-dimensional $\gl_n$-modules. We show that all such modules appear as restrictions of $\E_n$-modules. In particular, we obtain a combinatorial description of all Weyl type modules.
Note that the Weyl type modules are parameterized by a choice of color and two partitions one of which is colorless. These modules are
used in the theory of representations of $W$-algebras.
Namely, the $W$-algebra associated with $\mathfrak{gl}_n$ can be constructed by the quantum Drinfeld-Sokolov
reduction of $U\widehat{\mathfrak{gl}}_n$.
Irreducible representations of  this $W$-algebra with generic central charge are labeled by 
two dominant integral $\sln_n$ weights. These representations of the $W$-algebra are obtained by applying the Drinfeld-Sokolov reduction to the Weyl type modules.

\medskip 

As another application, we construct an $\E_n$-module $\mathcal H(u_1,\ldots,u_n)$ of level $q_1^{n/2}$ with a basis parameterized by $n$ partitions. After restriction to $U_q \widehat{\mathfrak{gl}}_n$, for generic $u_1,\dots,u_n$, the module $\mathcal H(u_1,\ldots,u_n)$
 becomes an irreducible Verma representation. We conjecture that the family of representations
$\mathcal H(u_1,\ldots,u_n)$ 
can be defined for all values of $u_1,\ldots,u_n$, and  that they give us a Wakimoto type
constructions of $U_q(\widehat{\mathfrak{gl}}_n)$ Feigin-Fuchs modules.

\medskip

Quasi-finite $\E_n$-modules also appear in geometry.
Using Nakajima's quiver construction, one can describe the $\E_n$ Fock  module 
in a space of equivariant $K$-theory of a moduli space of representations of some quiver.
The basis is given by the fixed points of the action of the torus on such manifolds, see \cite{N}.
Representations $\Hc(u_1,\ldots,u_n)$ geometrically appear by a very different way, \cite{FFNR}. One has to use the equivariant $K$-theory in the Laumon spaces. At the moment we are not aware of the geometric meaning of the Macmahon modules.

\medskip

Our paper is motivated by papers \cite{FFJMM1}, \cite{FFJMM2}, \cite{FJMM}, where we considered $\E_1$ modules. 
There are a lot of similarities in $\E_1$ and $\E_n$ representation theory but some features are different. The vector representation, Fock modules and Macmahon modules, all are present in $\E_1$ case and have bases labeled by the same combinatorial objects. 
However, the parameters $q_1,q_2,q_3$ are all interchangeable 
in $\E_1$ and we have Fock modules of levels $q_1^{1/2},q_2^{1/2}$ and $q_3^{1/2}$, while
in $\E_n$ there exist only Fock modules of level $q_2^{1/2}$. On the other hand, we 
have an extra $\Z^n$ grading (coloring) in $n>1$ case.
In particular,the condition that one of the partitions must be colorless 
is trivial in $\E_1$. 
Restriction to $U_q\widehat \gl_n$ subalgebras is more interesting if $n>1$. 
In general, we find that while $\E_1$ is a more symmetric algebra, it is easier to work with $\E_n$, $n\geq 3$ as Cartan type generators act by factored rational functions with fewer number of factors.

Note that as $q_i\to 1$ in an appropriate way, the algebra $\mathcal{E}_1$ becomes a $W_{1+\infty}$-algebra.
We expect that as $q_i\to 1$, the algebra $\mathcal{E}_n$ becomes the conformal algebra depending on parameter $N$.
When $N$ is a positive integer, this conformal algebra can be factorized to 
the coset $\widehat{\mathfrak{gl}}_N/\widehat{\mathfrak{gl}}_{N-n}$. 
Therefore, the geometric construction of the representations of $\mathcal{E}_n$ is expected to provide a ``geometric" construction of the space of states in some conformal field theories.
This phenomenon now is called  the AGT conjecture, and gives a big impact on
the activity in this direction.

\medskip 
 The parameters $q_1,q_2$ in the present paper are always generic, 
 we plan to discuss representations of $\E_n$ with non-generic values in a subsequent publication.

\medskip

The paper is organized as follows. We start with basic facts and definitions in Section \ref{tor alg section}, in particular we introduce horizontal and vertical $U_q\widehat{\gl}_n$ subalgebras. We construct the Fock representation in Section \ref{Fock sec}. Section \ref{Macmahon sec} is devoted to the construction of the general Macmahon module. In Section \ref{Weyl type sec}, we discuss the Weyl type $U_q\widehat{\gl}_n$ modules.

\section{Quantum toroidal $\gl_n$}\label{tor alg section}
\subsection{Algebra $\mathcal E_n^c$}
Fix a natural number $n\geq 3$. Let $I=\{1,\dots,n-1\}$ and $\hat I=\{0,1,\dots,n-1\}$.

Let $(a_{ij})_{i,j\in\hat I}$ be the Cartan matrix of type $A^{(1)}_{n-1}$. 
That is $a_{ii}=2$, $a_{i,i+1}=a_{i+1,i}=-1$, $i\in \hat I$, and $a_{m,n}=0$ 
otherwise. Let, in addition, $(m_{ij})_{i,j\in\hat I}$ be the matrix given by $m_{i-1,i}=-1$, 
$m_{i,i-1}=1$, $i\in \hat I$, and $m_{ij}=0$ otherwise.

Here and in many places below, we use the {\it cyclic modulo $n$} convention: an
index $i$ is identified with the index $n+i$. 
For example $m_{n,n-1}=m_{0,n-1}$, $a_{n-1,n}=a_{n-1,0}$, etc.
Also, for integers $a,b$ we write
\be
a\equiv b
\en
if and only if  $n$ divides $a-b$.

Let $d,q\in\C$ be non-zero complex parameters. 
The {\it quantum toroidal $\gl_n$} algebra, see \cite{GKV} 
is an associative algebra $\E_n^c=\E_n^c(q,d)$ with generators
$E_{i,k}$, $F_{i,k}$, $H_{i,l}$, $K_i^{\pm1}$, $q^{\pm \frac12 c}$ where
$k\in\Z$, $i\in \hat I$, $l\in\Z-\{0\}$ 
satisfying the following defining relations: $q^{\pm \frac12 c}$ are central,
\begin{gather*}
K_i K^{-1}_i = K^{-1}_i K_i=1, 
\quad 
q^{\frac{1}{2}c}q^{-\frac{1}{2}c}=q^{-\frac{1}{2}c}q^{\frac{1}{2}c}=1\,,\\
K^\pm_i(z)K^\pm_j (w) = K^\pm_j(w)K^\pm_i (z), 
\\
\frac{q^{a_{ij}}w -d^{m_{ij}}q^{-c}z}{q^{a_{ij}}w-d^{m_{ij}}q^cz} K^-_i(z)K^+_j (w) 
=
\frac{w -d^{m_{ij}}q^{a_{ij}}q^{-c}z}{w-d^{m_{ij}}q^{a_{ij}}q^cz}  K^+_j(w)K^-_i (z),
\\
(q^{a_{ij}}w - d^{m_{ij}}q^{\pm\frac{1}{2}c} z) K_i^\pm(z)E_j(w)= 
(w -d^{m_{ij}}q^{a_{ij}}q^{\pm\frac{1}{2}c}z) E_j(w)K_i^\pm(z),
\\
(w - d^{ m_{ij}}q^{\mp\frac{1}{2}c}q^{a_{ij}}z)K_i^\pm(z)F_j(w)
= (q^{a_{ij}}w -d^{ m_{ij}}q^{\mp\frac{1}{2}c}z) F_j(w)K_i^\pm(z),
\end{gather*}
\begin{gather*}
[E_i(z),F_j(w)]=\frac{\delta_{i,j}}{q-q^{-1}}(\delta\bigl(q^c\frac{w}{z}\bigr)K_i^+(q^{\frac{1}{2}c}w)
-\delta\bigl(q^c\frac{z}{w}\bigr)K_i^-(q^{\frac{1}{2}c}z)),\\
(d^{m_{ij}}z-q^{a_{ij}}w)E_i(z)E_j(w)=(d^{m_{ij}}q^{a_{ij}}z-w)E_j(w)E_i(z), \\
(d^{m_{ij}}z-q^{-a_{ij}}w)F_i(z)F_j(w)=(d^{m_{ij}}q^{-a_{ij}}z-w)F_j(w)F_i(z),\\
E_i(z_1)E_i(z_2)E_{i\pm1}(w)-(q+q^{-1})E_i(z_1)E_{i\pm 1}(w)E_i(z_2)+E_{i\pm1}(w)E_i(z_1)E_i(z_2)+ (z_1\leftrightarrow  z_2)=0,\\ 
F_i(z_1)F_i(z_2)F_{i\pm1}(w)-(q+q^{-1})F_i(z_1)F_{i\pm1}(w)F_i(z_2)+F_{i\pm1}(w)F_i(z_1)F_i(z_2)+ (z_1\leftrightarrow  z_2)=0,
\end{gather*}
for all values of $i,j\in\hat I$ and $[E_i(z),E_j(w)]=[F_i(z),F_i(w)]=0$ for $i\not\equiv j,j\pm 1$.

Here we collected the generators in the formal power series
\begin{align*}
E_i(z) =\sum_{k\in \Z}E_{i,k}z^{-k}, \quad 
F_i(z) =\sum_{k\in\Z}F_{i,k}z^{-k}, \quad
K_i^{\pm}(z) = K_i^{\pm 1} \exp(\pm(q-q^{-1})\sum_{k=1}^\infty H_{i,\pm k}z^{\mp k}),
\end{align*}
and $i,j\in\hat I$.

Let \be \kappa=\prod_{i\in\hat I} K_i. \en Then $\kappa$ is invertible and central.

\subsection{Horizontal and vertical $U_q\widehat \gl_n$}
In this section we define two subalgebras of $\E_n^c$ which we call horizontal and vertical $U_q\widehat\gl_n$.

The subalgebra of $\E_n^c$ generated by $E_{i,0}, F_{i,0},K_i^{\pm1}$, 
$i\in\hat I$ is called the {\it horizontal} quantum affine $\sln_n$ and denote it by 
$U_q^{hor}\widehat{\sln}_n$. 

The subalgebra of $\E_n^c$ generated by $E_i(z), F_i(z),K_i(z), q^{\pm \frac 12 c}$, $i\in I$ 
is called the {\it vertical} quantum affine $\sln_n$ and denote it by $U_q^{ver}\widehat{\sln}_n$. 

The vertical algebra is given in new Drinfeld realization and the
horizontal is in Drinfeld-Jimbo form.

We describe a Heisenberg subalgebra $\mathfrak{a}^{ver}$
of $\mathcal{E}_n^c$ 
commuting with the vertical algebra $U_q^{ver}\widehat{\mathfrak{sl}}_n$ as follows.

In terms of the generators $H_{i,r}$ we have the relations 
\begin{align*}
&[H_{i,r},E_j(w)]= \frac{[r a_{i,j}]}{r} d^{-r m_{i,j}}q^{-|r|\frac{c}{2}}\,w^r E_j(w)\,,\\
&[H_{i,r},F_j(w)]=-\frac{[r a_{i,j}]}{r} d^{-r m_{i,j}}q^{|r|\frac{c}{2}}\,w^r F_j(w)\,,\\
&[H_{i,r},H_{j,s}]=\delta_{r+s,0}\frac{[r a_{i,j}]}{r}
\frac{q^{rc}-q^{-rc}}{q-q^{-1}}  
d^{-r m_{i,j}}\,,
\end{align*}
where $r,s\neq 0$ and $[x]=(q^x-q^{-x})/(q-q^{-1})$. 
For each $r\neq 0$, let $\{c_{i,r}\}_{i=0}^{n-1}$ be a non-trivial solution of the equation
\begin{align*}
\sum_{i=0}^{n-1}c_{i,r}[r a_{i,j}] d^{-r m_{i,j}}=0
\quad (j=1,\ldots,n-1).
\end{align*}

Let $\mathfrak{a}^{ver}$ be the subalgebra of $\mathcal{E}_n$ 
generated by $H^{ver}_r=\sum_{i=0}^{n-1}c_{i,r}H_{i,r}$, $r\in\Z_{\neq0}$. 
Clearly $\mathfrak{a}^{ver}$ is a Heisenberg subalgebra with central element $q^c$ which
commutes with $U^{ver}_q\widehat{\mathfrak{sl}}_n$. 

We call the subalgebra of $\E^c_n$ generated by $U_q^{ver}\widehat{\sln}_n$ and  $\mathfrak{a}^{ver}$ the {\it vertical 
quantum affine $\gl_n$} and denote it by 
$U_q^{ver}\widehat{\gl}_n$. 

\medskip

To construct the horizontal counterpart of $U_q^{ver}\widehat{\gl}_n$ we use an automorphism described in \cite{M99}.
Define the principal grading of $\mathcal{E}_n^c$ by 
\be
\mathrm{pdeg}\ E_{i,k}=1\,,
\quad \mathrm{pdeg}\ F_{i,k}=-1\,,
\quad \mathrm{pdeg}\ H_{i,k}=0,\, \quad \mathrm{pdeg}\ K_i^{\pm1}=0.
\en

\begin{prop}\cite{M99}
There exists an automorphism $\theta$ of $\mathcal{E}_n^c$ such that
\be
\theta(U_q^{ver}\widehat{\sln}_n)=U_q^{hor}\widehat{\sln}_n, \qquad \theta (U_q^{hor}\widehat{\sln}_n)=U_q^{ver}\widehat{\sln}_n,
\en
Moreover,
\be
\theta (\kappa)=q^{c},\qquad \theta(q^c)=\kappa^{-1}.
\en
We also have
\be \mathrm{pdeg}\ \theta (H_{i,k})=nk. \en
\end{prop}

We fix the automorphism $\theta$ as in \cite{M99} (it is called $\psi$ there).
Let $\mathfrak{a}^{hor}=\theta \ \mathfrak{a}^{ver}$ be the image of $\mathfrak{a}^{ver}$. Then $\mathfrak{a}^{hor}$ is a 
Heisenberg algebra with central element $\kappa$ which commutes with $U_q^{hor}\widehat{\sln}_n$. 

We call the subalgebra of $\E_n^c$ generated by $U_q^{hor}\widehat{\sln}_n$ and  $\mathfrak{a}^{hor}$ the {\it horizontal 
quantum affine $\gl_n$} and denote it by 
$U_q^{hor}\widehat{\gl}_n$.

\subsection{Basic properties}
Let $\E_n$ be the quotient algebra of $\E^c_n$ by the relation $q^{\pm\frac12 c}=1$. 
In this paper we study only representations of algebra $\E_n$ and this is the algebra we employ in the rest of the paper.

By abuse of language we use the same notation for generators and subalgebras of $\E_n$ and $\E_n^c$: $E_i(z)$, $U_q^{ver}\widehat{\gl}_n$, etc.

In this section we give standard definitions and collect
simple facts about $\E_n$ and its modules.

\medskip

The element $\kappa$ is the central element of the $U_q^{hor}\widehat{\gl}_n$. 
The vertical algebra $U_q^{ver}\widehat{\gl}_n$ is the loop algebra without central extension.

\medskip

The algebra $\E_n$  is $\Z^{n}$ graded by the rule
\be
\deg(E_i(z))=1_i,\qquad \deg(F_i(z))=-1_i,\qquad \deg(K^\pm_i(z))=0,
\en
here $1_i$ are the standard generators of $\Z^n$. 

Note that there exists also the homogeneous $\Z$-grading of $\E_n$ given by
\be
\deg_\delta E_{i,k}=nk\,,
\quad \deg_\delta F_{i,k}=nk\,,
\quad \deg_\delta H_{i,k}=nk\, .
\en
However, we do not consider homogeneous grading in this paper, moreover, all modules we consider
will be $\Z^n$ graded but not homogeneously graded. 

Let $\E_n^+$, $\E_n^-$, $\E_n^0$ be the subalgebras of $\E_n$ generated by coefficients of the series $E_i(z)$, $i\in\hat I$, $F_i(z)$, $i\in\hat I$, and $K_i(z)$, $i\in\hat I$, respectively. 

\begin{lem} We have the triangular decomposition
$
\E_n=\E^-_n\otimes\E^0_n\otimes\E^+_n.
$ \qquad $\Box$
\end{lem}

\medskip

Let $\omega:\ \E_n(q,d)\to \E_n(q^{-1},d)$ be the map 
sending the series 
$E_i(z), F_i(z),
K_i^{\pm}(z)$ to $F_{i}(z)$, $E_{i}(z)$, $K_{i}^{\pm}(z)$, respectively,
$i\in\hat I$. Clearly, $\omega$ extends to a graded isomorphism of algebras.

\medskip

Let $\tau:\ \E_n(q,d)\to \E_n(q,d)$ be the map sending the series 
$E_i(z), F_i(z),
K_i^{\pm}(z)$ to $E_{i+1}(z)$, $F_{i+1}(z)$, $K_{i+1}^{\pm}(z)$, respectively,
$i\in\hat I$. Clearly, $\tau $ extends to an automorphism  of $\E_n$ of order $n$: $\tau^n=id$.

\medskip

For $a\in\C^\times$, let $s_a:\ \E_n(q,d)\to \E_n(q,d)$ be the map sending the series 
$E_i(z), F_i(z),
K_i^{\pm}(z)$ to $E_{i}(az), F_{i}(az), K_{i}^{\pm}(az)$, respectively,
$i\in\hat I$. Clearly, $s_a $ extends to a graded automorphism of 
$\E_n$. 

\medskip

Let $\iota:\ \E_n(q,d)\to \E_n(q,d^{-1})$ be the map 
sending the series 
$E_i(z), F_i(z),
K_i^{\pm}(z)$ to $E_{n-i}(z)$, $F_{n-i}(z)$, $K_{n-i}^{\pm}(z)$, respectively,
$i\in\hat I$. Clearly, $\iota$ extends to an isomorphism of algebras.

\medskip

We use the following notation:
\be
q_1=d/q,\quad q_2=q^2,\quad q_3=1/(dq).
\en
Note that $q_1q_2q_3=1$. Using the automorphism $\iota$,  one can
switch $q_1$ and $q_3$ keeping $q_2$ unchanged.
 
In this paper we assume that $q_1,q_2,q_3$ are {\it generic}. By that we mean that if $q_1^aq_2^bq_3^c=1$ for some integers $a,b,c$ then $a=b=c$. In particular, 
$q_1,q_2,q_3$ are not roots of unity.

\medskip

Let $V$ be an $\E_n$-module.

A vector $v\in V$ is called {\it singular} if $\E^-v=0$. 

Let $\bs \phi^\pm(z)=(\phi_i^\pm(z))_{i\in\hat I}$ and 
$\phi_i^\pm(z)\in\C[[z^{\mp1}]]$. A vector $v\in V$ has weight $\bs\phi^\pm(z)$ if $K_i^\pm(z)v=\phi_i^\pm(z)v$, $i\in\hat I$.

The module $V$ is called {\it weighted} if the commuting series $K_i^\pm(z)$, $i\in\hat I$ are diagonalizable in $V$. The module $V$ is called {\it tame} if it is weighted and the joint spectrum of $K_i^\pm(z)$, $i\in\hat I$, is simple.

The module $V$ is called {\it lowest weight module} with {\it lowest weight} $\bs\phi^\pm(z)$ if it is generated by a singular 
weight vector $v$ of the weight $\bs\phi^\pm(z)$. In such a case $v$ is called the {\it lowest weight vector}. 

The weighted module $V$ is called {\it quasi-finite} if for all sequences of complex numbers $(a_i)_{i\in\hat I}$, $a_i\in\C$, we have 
$\dim\{v\in V\ |\ K_i v=a_i v\ (i\in\hat{I})\}<\infty$.

The module $V$ is called {\it level K} module if the central element $\kappa^{-1}=\prod_{i\in\hat I}K_i^{-1}$ acts in $V$ by the constant $K$.

\begin{thm}\label{quasifin}
For any $\bs \phi^\pm(z)$ such that $\phi_i^+(\infty)\phi_i^-(0)=1$, there exists a unique up to isomorphism irreducible lowest weight $\E_n$-module $L_{\bs \phi^\pm(z)}$ with lowest weight $\bs \phi^\pm(z)$. The module $L_{\bs\phi^\pm(z)}$ is weighted and $\Z^n$-graded.

Let $q_1,q_2,q_3$ be generic. Then module $L_{\bs \phi^\pm(z)}$ is quasi-finite if and only if for each $i\in\hat I$, the series $\phi^\pm_i(z)$ are expansions of a rational function $\phi_i(z)$ such that $\phi_i(\infty)\phi_i(0)=1$.
\end{thm}
\begin{proof}
The theorem is standard, the proof is similar for example to the proof of Theorem 6.1 in \cite{M}.
\end{proof}

\subsection{Comultiplication}

Let $\Delta$ be the comultiplication map given by:
\begin{align*}
\Delta(K_i^\pm(z))&=K_i^\pm(z)\otimes K_i^\pm (z), \\
\Delta(E_i(z))&=E_i(z)\otimes 1 +K_i^-(z)\otimes E_i(z), \\
\Delta(F_i(z))&= F_i(z)\otimes K_i^+(z)+1\otimes F_i(z),
\end{align*}
for $i\in\hat I$. Note that $\Delta$ is not an honest map as $\Delta(E_i(z))$ and $\Delta(F_i(z))$ 
contains infinite summations. However, if the summation is well-defined in a tensor product of $\E_n$-modules, it can be used. 
Moreover, in some cases $\Delta$ can be used even if it is not well-defined on the whole tensor product. Namely, we have the following lemma.

\begin{lem}\label{delta sub}
Let $V_1$, $V_2$ be $\E_n$-modules. Let $U\subset V_1\otimes V_2$ be a linear subspace such that 
for any $u\in U$, the coefficients of series $\Delta(K_i(z))u$, $\Delta(E_i(z))u$,
$\Delta(F_i(z))u$, $i\in\hat I$ are well-defined vectors in $U$. Then $U$ has an $\E_n$-module structure such that the series $K_i(z)$, $E_i(z)$, $F_i(z)$ act as $\Delta(K_i(z))u$, $\Delta(E_i(z))u$, $\Delta(F_i(z))u$, respectively.
\end{lem}
\begin{proof}
The lemma is proved by the straightforward formal check of compatibility 
of $\Delta$ and relations in $\E_n$. 
\end{proof}

Lemma \ref{delta sub} deals with the submodule situation. We also have the following useful quotient-module version.

Suppose we have the decomposition of the vector space $V_1\otimes
V_2=U\oplus W$ such that $\Delta(K_i(z))U\subset U$ and
$\Delta(K_i(z))W\subset W$, $i\in\hat I$. Let $\on{pr}_W:\ U\oplus
W\to W$ and $\on{pr}_U:\ U\oplus W\to U$ denote the projections along
$U$ and $W$ respectively.

Assume that the coefficients of the series $\Delta(E_i(z))$ and 
$\Delta (F_i(z))$ are well-defined operators when acting on vectors in $U$.
We do not require that the result is contained in $U$.
Assume further that the vector spaces $U$ and $W$ have bases $\{u_j\}_{j\in J}$ and $\{w_k\}_{k\in K}$ respectively such that the matrix coefficients of operators  $\Delta(E_i(z))$ and $\Delta(F_i(z))$ applied to $w_k$ and computed on $u_j$ are zero for all $k\in K,j\in J$. Note that we do not require
the finiteness of the coefficients of $w_k$.
  
\begin{lem}\label{delta fac} Under the assumption above, the space $U$ has an $\E_n$-module structure such that the series $K_i(z)$, $E_i(z)$, $F_i(z)$ act by $\Delta(K_i(z))$, $pr_U\Delta(E_i(z))$, $pr_U\Delta(F_i(z))$, respectively. \qquad $\Box$.
\end{lem}

\section{The Fock representations of $\E_n$}\label{Fock sec}
\subsection{The vector representations}
Fix $k\in\hat I$.
Let $u\in\C^\times$ be a non-zero complex number and 
let $V^{(k)}(u)$ be a complex vector space with basis $[u]_j^{(k)},$ $j\in\Z$.
In this notation index $j$ is {\it not} considered modulo $n$.

We define the $\Z^n$ grading on $V^{(k)}(u)$ by setting
\be
\deg[u]_j^{(k)}=m(1_0+\dots+1_{n-1})+1_k+1_{k-1}+\dots+1_{k-r},
\en
if $j=mn+r$ and $r\in\{0,1,\dots,n-1\}$. Here $1_i$ denote the standard generators of $\Z^n$ 
and we continue to use the cyclic conventions for indexes.

In particular, we have $\deg[u]_{-1}^{(k)}=0$, $\deg[u]_0^{(k)}=1_k$, 
$\deg[u]_j^{(k)}-\deg[u]_{j-1}^{(k)}=1_{k-j}$. 

We picture $[u]_j^{(k)}$ as a semi-infinite row of boxes ending at box $j+1$ on the right of the divider.
We color the boxes in colors $k,k-1,\dots,1,0,n-1,n-2\dots$ from the divider to the right and continuing the pattern to the left, 
see Figure 1. 

Let 
\be
\psi(z)=\frac{q-q^{-1}z}{1-z}.
\en
We have $\psi(1/z)=\frac{q^{-1}-qz}{1-z}=(\psi(q_2z))^{-1}$.

\medskip

Define the action of operators $K_i^\pm(z)$, $E_i(z)$, $F_i(u)$, $i\in\hat I$ by the formulas
\begin{align}
&E_i(z)[u]_j^{(k)}=\begin{cases}\delta(q_1^{j+1}u/z)[u]_{j+1}^{(k)}, \qquad &i+j+1\equiv k;\\
0,\qquad & i+j+1\not\equiv k;
\end{cases}\notag\\
&F_i(z)[u]_{j+1}^{(k)}=\begin{cases}\delta(q_1^{j+1}u/z)[u]_j^{(k)}, \qquad &i+j+1\equiv k; \\
0,\qquad & i+j+1\not\equiv k;
\end{cases}
\label{vector}\\
&K^\pm_i(z)[u]_j^{(k)}=\begin{cases} \psi(q_1^ju/z)[u]_j^{(k)}, \qquad &j+i\equiv k;\\
\psi(q_1^jq_3^{-1}u/z)^{-1} [u]_j^{(k)} \qquad &j+i+1\equiv k;\\
 [u]_j^{(k)},\qquad & \on{otherwise.}
\end{cases} \notag
\end{align}

The following lemma is checked by a straightforward calculation.

\begin{lem} Formulas \Ref{vector} define a structure of an
irreducible, tame, $\Z^n$-graded $\E_n$-module of level $1$ on the space $V^{(k)}(u)$. \qquad $\Box$
\end{lem}
We call the module $V^{(k)}(u)$ the {\it vector} representation. 
Note that the 
operator $E_i(z)$ adds
boxes of color $i$ and operator $F_i(z)$ removes
boxes of color $i$.

Note, that the parameter $u$ can be changed by using the shift of 
evaluation parameter, $s_a$. Twisting the module $V^{(k)}(u)$ by the automorphism $\tau^i$, 
we obtain the module $V^{(k-i)}(u)$.

Another set of vector representations $\bar V^{(n-k)}(u;q,d^{-1})$ is obtained by 
twisting the module $V^{(k)}(u)$ with map $\iota$. 
We denote the $\E_n$-representation $\bar V^{(k)}(u;q,d)$ simply by 
$\bar V^{(k)}(u)$. We denote the basis of  $\bar V^{(k)}(u)$ by  $\bar{[u]}_j^{(k)}$. We change the labeling of the basic vectors so that
$\bar{[u]}_j^{(k)}$ corresponds to the vector $[u]^{(n-k)}_{-j-1}$.
For example, we have: 
\be
E_i(z) \bar{[u]}_j^{(k)}=\begin{cases}\delta(q_3^{j+1}u/z)\bar{[u]}_{j+1}^{(k)}, \qquad &i \equiv k+j+1;\\
0,\qquad & i\not\equiv k+j+1.\end{cases}
\en

In $\bar V^{(k)}(u)$, 
we picture $\bar{[u]}_j^{(k)}$ as a
semi-infinite column of boxes ending at box $j+1$ to the bottom of the divider.
We color the boxes in colors $k,k+1,\dots,n-1,0,1,\dots$ from the divider to the bottom and continuing the pattern up, 
see Figure 1.  
Then, as before, operators $E_i(z)$ add boxes of color $i$ and operators $F_i(z)$ remove boxes of color $i$.  

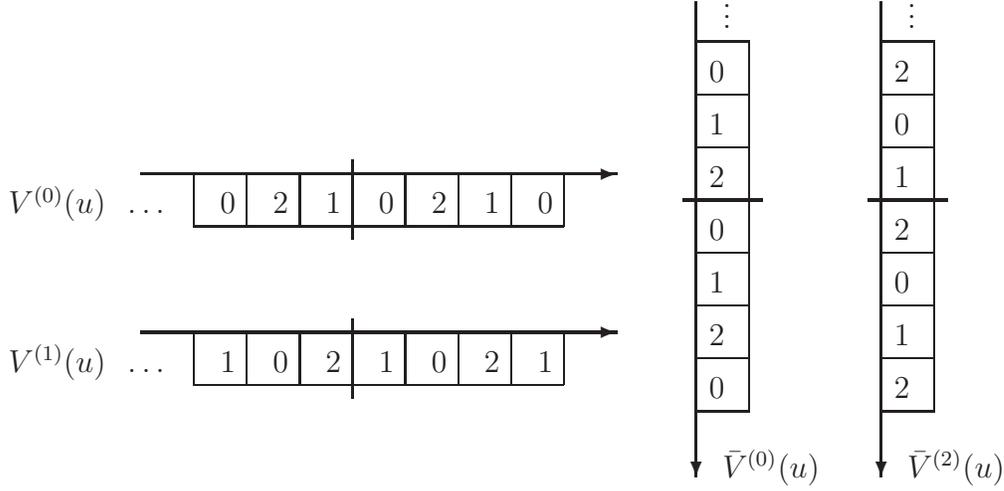
\begin{figure}
\begin{picture}(150,200)(40,00)\label{fig}

\put(-30,130){\line(1,0){140}}
\put(-30,110){\line(1,0){140}}

\put(110,130){\line(0,-1){20}}
\put(90,130){\line(0,-1){20}}
\put(70,130){\line(0,-1){20}}
\put(50,130){\line(0,-1){20}}
\put(30,130){\line(0,-1){20}}
\put(10,130){\line(0,-1){20}}
\put(-10,130){\line(0,-1){20}}
\put(-30,130){\line(0,-1){20}}

\put(-20,115){$0$}
\put(0,115){$2$}
\put(20,115){$1$}
\put(40,115){$0$}
\put(60,115){$2$}
\put(80,115){$1$}
\put(100,115){$0$}

\put(-55,115){$\dots$}
\put(-100,115){$V^{(0)}(u)$}

\put(-30,50){\line(1,0){140}}

\put(110,70){\line(0,-1){20}}
\put(90,70){\line(0,-1){20}}
\put(70,70){\line(0,-1){20}}
\put(50,70){\line(0,-1){20}}
\put(30,70){\line(0,-1){20}}
\put(10,70){\line(0,-1){20}}
\put(-10,70){\line(0,-1){20}}
\put(-30,70){\line(0,-1){20}}

\put(-20,55){$1$}
\put(0,55){$0$}
\put(20,55){$2$}
\put(40,55){$1$}
\put(60,55){$0$}
\put(80,55){$2$}
\put(100,55){$1$}

\put(-55,55){$\dots$}
\put(-100,55){$V^{(1)}(u)$}

\put(170,185){$\vdots$}
\put(160,180){\line(0,-1){140}}
\put(180,180){\line(0,-1){140}}

\put(160,180){\line(1,0){20}}
\put(160,160){\line(1,0){20}}
\put(160,140){\line(1,0){20}}
\put(160,120){\line(1,0){20}}
\put(160,100){\line(1,0){20}}
\put(160,80){\line(1,0){20}}
\put(160,60){\line(1,0){20}}
\put(160,40){\line(1,0){20}}

\put(165,165){$0$}
\put(165,145){$1$}
\put(165,125){$2$}
\put(165,105){$0$}
\put(165,85){$1$}
\put(165,65){$2$}
\put(165,45){$0$}

\put(170,15){$\bar V^{(0)}(u)$}

\put(240,185){$\vdots$}
\put(230,180){\line(0,-1){140}}
\put(250,180){\line(0,-1){140}}

\put(230,180){\line(1,0){20}}
\put(230,160){\line(1,0){20}}
\put(230,140){\line(1,0){20}}
\put(230,120){\line(1,0){20}}
\put(230,100){\line(1,0){20}}
\put(230,80){\line(1,0){20}}
\put(230,60){\line(1,0){20}}
\put(230,40){\line(1,0){20}}

\put(235,165){$2$}
\put(235,145){$0$}
\put(235,125){$1$}
\put(235,105){$2$}
\put(235,85){$0$}
\put(235,65){$1$}
\put(235,45){$2$}

\put(240,15){$\bar V^{(2)}(u)$}

\thicklines
\put(-50,130){\vector(1,0){180}}
\put(-50,70){\vector(1,0){180}}
\put(160,195){\vector(0,-1){180}}
\put(30,75){\line(0,-1){30}}
\put(30,135){\line(0,-1){30}}
\put(155,120){\line(1,0){30}}
\put(230,195){\vector(0,-1){180}}
\put(225,120){\line(1,0){30}}

\end{picture}
\caption{Picturing $[u]_{3}^{(0)}$, $[u]_{3}^{(1)}$, 
$\bar{[u]}_{-4}^{(0)}$,  $\bar{[u]}_{-4}^{(2)}$ with $n=3$. The dividers are the think longer lines.}
\end{figure}

\subsection{Tensor product of vector representations}
We consider the tensor product of vector spaces
$V^{(k)}(u)\otimes V^{(l)}(v)$. The basis of 
$V^{(k)}(u)\otimes V^{(l)}(v)$ is given by 
$[u]^{(k)}_i\otimes [v]^{(l)}_j$,    
$i,j\in\Z$.
We use the comultiplication to define the $\E_n$-module structure.
The situation is described in the following lemma.

\begin{lem} \label{tensor two} We have the following cases.
\begin{enumerate}
\item If $uq_1^{nm+k-l}=v$ 
for some $m\in\Z$, the action of $\E_n$ is not well-defined. In all other cases the action of $\E_n$ is well-defined.
\item If $uq_2q_1^{nm+k-l}=v$ for some $m\in\Z$, then  the module $V^{(k)}(u)\otimes V^{(l)}(v)$ 
has a submodule spanned by $[u]^{(k)}_i\otimes [v]^{(l)}_j$, with $i\geq j+nm+k-l$. 
The submodule and the quotient module are tame, irreducible $\E_n$-modules of level $1$.
\item If $uq_2^{-1}q_1^{nm+k-l}=v$ for some $m\in\Z$, 
then  the module $V^{(k)}(u)\otimes V^{(l)}(v)$ 
has a submodule spanned by $[u]^{(k)}_i\otimes [v]^{(l)}_j$, with $i\geq j+nm+k-l+1$. 
The submodule and the quotient module are tame, 
irreducible $\E_n$-modules of level $1$.
\item In all other cases the module $V^{(k)}(u)\otimes V^{(l)}(v)$ 
is a tame, irreducible $\E_n$-module of level $1$.
\end{enumerate}
\end{lem}
\begin{proof} Consider $E_s(z)[u]^{(k)}_i\otimes [v]^{(l)}_j$. It 
equals a linear combination of $[u]^{(k)}_{i+1}\otimes [v]^{(l)}_j$ and 
$[u]^{(k)}_i\otimes [v]^{(l)}_{j+1}$. The first vector has a well defined 
coefficient which is nonzero if and and only if 
$s+i\equiv k-1$. 

The second coefficient is zero unless $s+j\equiv l-1$. 
If $s+j\equiv l-1$ then if $s+i\not\equiv k,k-1$ the coefficient is well-defined
and non-zero since $K_s(z)[u]^{(k)}_i=[u]^{(k)}_i$. 
In the case $s+i\equiv k$ the coefficient equals:
\be
\delta(q_1^{j+1}v/z)\frac{q-q^{-1}q_1^iu/z}{1-q_1^iu/z}.
\en
It is undefined if and only if $v/u=q_1^{i-j-1}$ and it is zero if and only if
$v/u=q_2^{-1}q_1^{i-j-1}$. Note that $i-j-1\equiv k-l$. 

The case $s+i\equiv k-1$ is considered similarly. In this case the second coefficient has the form
\be
\delta(q_1^{j+1}v/z)\frac{q^{-1}-q q_1^{i+1}u/z}{1-q_1^{i+1}u/z}.
\en

The action of operator $F_s(z)$ is treated in the same way.
The lemma follows.
\end{proof}

One can similarly describe other cases of the tensor product 
of two vector representations: 
$V^{(k)}(u)\otimes \bar V^{(l)}(u)$, $\bar V^{(k)}(u)\otimes V^{(l)}(u)$, 
$\bar V^{(k)}(u)\otimes \bar V^{(l)}(u)$.

\subsection{Fock space}
We construct the Fock representation by the inductive procedure 
(semi-infinite wedge construction) from the vector representation.

Fix $k\in\hat I$.
Consider the following tensor product of $r$ vector representations:
\be
V^{(k)}(u)\otimes V^{(k)}(uq_2^{-1})\otimes V^{(k)}(uq_2^{-2})\otimes \dots \otimes V^{(k)}(uq_2^{-r+1}).
\en
By Lemma \ref{tensor two} this tensor product has a submodule $W^{(k)}_r$ spanned by vectors
\be
\ket{\la}=[u]^{(k)}_{\la_1-1}\otimes [uq_2^{-1}]^{(k)}_{\la_2-2}\otimes [uq_2^{-2}]^{(k)}_{\la_3-3}\otimes \dots \otimes
[uq_2^{-r+1}]^{(k)}_{\la_r-r}, 
\en
where $\la_i$ are arbitrary integers satisfying 
$\la_1\geq\la_2\geq \la_3\geq \dots\geq \la_r$. 
Next we argue that as $r\to\infty$ one can define the limit of $W_r$.

\medskip

Let $\F^{(k)}(u)$ be the subset of the infinite tensor product
\be
\F^{(k)}(u)\subset V^{(k)}(u)\otimes V^{(k)}(uq_2^{-1})\otimes V^{(k)}(uq_2^{-2})\otimes \dots,
\en
spanned by vectors
\be
\ket{\la}=[u]_{\la_1-1}^{(k)}\otimes [uq_2^{-1}]_{\la_2-2}^{(k)}\otimes [uq_2^{-2}]_{\la_3-3}^{(k)}\otimes\dots , 
\en
where $\la$ is a partition: $\la_i$ are arbitrary non-negative integers such that only finitely many are nonzero and $\la_1\geq\la_2\geq \la_3\geq \dots$. 

Given $\la=(\la_1,\la_2,\dots)$ and $j\in\Z_{\geq 1}$ let $\la\pm\bs 1_j=(\la_1,\la_2,\dots,\la_j\pm 1,\dots)$.

We define the grading on $\F^{(k)}(u)$ by the formula
\be
\deg \ket{\la}=\sum_{r=1}^\infty (\deg [u]_{\la_r-r}^{(k)}-\deg [u]_{-r}^{(k)}).
\en
Note that the sum is finite due to the stabilization of $\la_i$ to zero.

Define the action of the $\E_n$ on $\F^{(k)}(u)$ by the following formulas.

For $i\in \hat I$, $j\in\Z_{\geq 1}$ such that $i+\la_j+1=j+k$, set
\be
\bra{\la+{\bf 1}_j}E_i(z) \ket{\la}
=\hspace{-10pt}\prod_{\substack{s=1, \\   \la_s+i\equiv s+k}}^{j-1} \hspace{-10pt}
 \psi(q_1^{\la_s-\la_j-1}q_3^{s-j})
\prod_{\substack{s=1,\\ \la_s+i+1\equiv s+k}}^{j-1} \hspace{-10pt}\psi(q_1^{\la_j-\la_s}q_3^{j-s})
\ \ \delta(q_1^{\la_j}q_3^{j-1}u/z),
\en

\be
\bra{\la}F_i(z) \ket{\la+{\bf 1}_j}
=\hspace{-10pt}\prod_{\substack{s=j+1, \\   \la_s+i\equiv s+k}}^{\infty} \hspace{-10pt}
 \psi(q_1^{\la_s-\la_j-1}q_3^{s-j})
\prod_{\substack{s=j+1,\\ \la_s+i+1\equiv s+k}}^{\infty} \hspace{-10pt}\psi(q_1^{\la_j-\la_s}q_3^{j-s})
\ \ \delta(q_1^{\la_j}q_3^{j-1}u/z).
\en

For $i\in \hat I$, set
\bea\label{K act}
\bra{\la}K_i^{\pm}(z) \ket{\la}
=\hspace{-10pt}\prod_{\substack{s=1, \\   \la_s+i\equiv s+k}}^{\infty} \hspace{-10pt}
 \psi(q_1^{\la_s-1}q_3^{s-1}u/z)
\prod_{\substack{s=1,\\ \la_s+i+1\equiv s+k}}^{\infty} \hspace{-10pt}
\psi(q_1^{\la_s-1}q_3^{s-2}u/z)^{-1}.
\ena
We set all other matrix coefficients to be zero.

Here we used the bra-ket notation for the matrix elements of the linear operators acting in $\F^{(k)}(u)$ in basis $\ket{\la}$.

Note that if $\la+\bs 1_j$ is not a partition then the above definition automatically
gives $\bra{\la+{\bf 1}_j}E_i(z) \ket{\la}=0$. Similarly, if $\la+\bs 1_j$ is a partition and $\la$ is not, then the above definition gives $\bra{\la}F_i(z) \ket{\la+{\bf 1}_j}=0$.

\medskip

Note that even though the formulas for the operators 
$F_i(z)$ and $K_i^{\pm}(z)$ involve infinite products, we mean only finite products.
Namely, we have the identity
\be
\psi(q_2z)\psi(1/z)=1,
\en
and we understand that once $\la_r=0$ and $j\equiv r+k$,
\be
\hspace{-10pt}\prod_{\substack{s=r, \\   \la_s+j\equiv s+k}}^{\infty} \hspace{-13pt}
 \psi(q_1^{\la_s-\la_i-1}q_3^{s-i})
\prod_{\substack{s=r,\\ \la_s+j+1\equiv s+k}}^{\infty} \hspace{-10pt}\psi(q_1^{\la_i-\la_s}q_3^{i-s})=\hspace{-0pt}\prod_{t=0}^\infty \hspace{-0pt}
 \psi(q_1^{-\la_i-1}q_3^{r+nt-i})\psi(q_1^{\la_i}q_3^{i-r-nt-1})=1,
\en 
and, similarly,
\be
\hspace{-5pt}\prod_{\substack{s=r, \\   \la_s+j\equiv s+k}}^{\infty} \hspace{-10pt}
 \psi(q_1^{\la_s-1}q_3^{s-1}u/z)
\prod_{\substack{s=r,\\ \la_s+j+1\equiv s+k}}^{\infty} \hspace{-10pt}
\psi(q_1^{\la_s-1}q_3^{s-2}u/z)^{-1}=1.
\en

\medskip

\begin{prop}
The formulas above define an action of $\E_n$ on $\F^{(k)}(u)$ making it an irreducible, tame, lowest weight $\Z^n$-graded module of level $q$ and lowest weight
\begin{align*}
\bs \phi(z)=\(\psi(q_2u/z)^{-\delta_{i,k}}\)_{i\in\hat I}.
\end{align*}
\end{prop}
\begin{proof}
We have to check that the action of $E_j,F_j,K_j$, $j\in\hat I$, satisfy the 
relations of the quantum toroidal algebra $\E_n$. For example, let us consider the commutator $[E_j(z),F_j(w)]$ applied to $\ket{\la}$. Choose $r$ such that $\la_r=0$ and $j\equiv r+1$. 
Then the action of $[E_j(z),F_j(w)]$ on $\ket{\la}$ in $\F^{(k)}(u)$ coincides with the action in $W^{(k)}_{r-1}$. The same is true for the action of $K_j^{\pm}(z)$ on $\ket{\la}$. Therefore the corresponding relation holds. The other relations follow for the same reason.

The module $\F^{(k)}(u)$ is clearly tame because knowing 
the eigenvalue of $K_j(z)$, $j\in\hat I$ on $\ket{\la}$ we can uniquely identify $\la_i$ as 
follows.
We start by identifying the pole at $1-q_1^{\la_1}u/z=0$. After $\la_1,\dots,\la_{i-1}$ are determined, we find $\la_i$ from the pole at $1-q_1^{\la_i}q_3^{s-1}u/z=0$. Note that 
if such a pole exists it is unique and if it does not then $\la_i=\la_{i-1}$.

It follows that the representation is irreducible.

The lowest weight vector is given by 
\be
\ket{\emptyset}=[u]_{-1}^{(k)}\otimes [uq_2^{-1}]_{-2}^{(k)}\otimes [uq_2^{-2}]_{-3}^{(k)}\otimes\dots , 
\en
and the lowest weight is computed by a straightforward computation.
\end{proof}

We picture the vector $\ket{\la}$ using the Young diagram of $\la$. In addition, we assign to each box the color as before, so that operators $E_i(z)$ (resp. $F_i(z)$) add (resp. remove) boxes of color $i$. Note that the top left corner of $\F^{(k)}(u)$ has color $k$, see Figure 2.

Note that the automorphism $\tau$ relates $\F^{(k)}(u)$ and $\F^{(k-1)}(u)$ and the automorphism $s_a$ relates the modules $\F^{(k)}(u)$ and $\F^{(k)}(u/a)$. In contrast to vector representations,
the map $\iota$ does not produce any new Fock spaces. 

We remark that the module $\F^{(k)}(u)$ can be constructed in a similar way using the infinite tensor product
\be
\bar V^{(k)}(u)\otimes \bar V^{(k)}(uq_2^{-1})\otimes \bar V^{(k)}(uq_2^{-2})\otimes \dots \ .
\en
The two ways of constructing the Fock space correspond to``slicing" the Young diagram of a partition into rows and columns respectively.

\begin{figure}
\begin{picture}(150,130)(80,00)

\put(-10,110){\line(1,0){140}}
\put(-10,90){\line(1,0){140}}

\put(130,110){\line(0,-1){20}}
\put(110,110){\line(0,-1){20}}
\put(90,110){\line(0,-1){20}}
\put(70,110){\line(0,-1){20}}
\put(50,110){\line(0,-1){20}}
\put(30,110){\line(0,-1){20}}
\put(10,110){\line(0,-1){20}}
\put(-10,110){\line(0,-1){20}}

\put(0,95){$0$}
\put(20,95){$2$}
\put(40,95){$1$}
\put(60,95){$0$}
\put(80,95){$2$}
\put(100,95){$1$}
\put(120,95){$\underline 0$}
\put(140,95){$\bar 2$}

\put(-10,70){\line(1,0){80}}

\put(70,90){\line(0,-1){20}}
\put(50,90){\line(0,-1){20}}
\put(30,90){\line(0,-1){20}}
\put(10,90){\line(0,-1){20}}
\put(-10,90){\line(0,-1){20}}

\put(0,75){$1$}
\put(20,75){$0$}
\put(40,75){$2$}
\put(60,75){$\underline 1$}
\put(80,75){$\bar 0$}

\put(-10,50){\line(1,0){40}}

\put(30,70){\line(0,-1){20}}
\put(10,70){\line(0,-1){20}}
\put(-10,70){\line(0,-1){20}}

\put(0,55){$2$}
\put(20,55){$1$}
\put(40,55){$\bar 0$}

\put(-10,50){\line(1,0){40}}
\put(-10,30){\line(1,0){40}}

\put(30,50){\line(0,-1){20}}
\put(10,50){\line(0,-1){20}}
\put(-10,50){\line(0,-1){20}}

\put(0,35){$0$}
\put(20,35){$\underline 2$}

\put(0,15){$\bar 1$}

\put (50, 10){$\F^{(0)}(u)$}

\put(170,110){\line(1,0){140}}
\put(170,90){\line(1,0){140}}

\put(310,110){\line(0,-1){20}}
\put(290,110){\line(0,-1){20}}
\put(270,110){\line(0,-1){20}}
\put(250,110){\line(0,-1){20}}
\put(230,110){\line(0,-1){20}}
\put(210,110){\line(0,-1){20}}
\put(190,110){\line(0,-1){20}}
\put(170,110){\line(0,-1){20}}

\put(180,95){$2$}
\put(200,95){$1$}
\put(220,95){$0$}
\put(240,95){$2$}
\put(260,95){$1$}
\put(280,95){$0$}
\put(300,95){$\underline 2$}
\put(320,95){$\bar 1$}

\put(170,70){\line(1,0){80}}

\put(250,90){\line(0,-1){20}}
\put(230,90){\line(0,-1){20}}
\put(210,90){\line(0,-1){20}}
\put(190,90){\line(0,-1){20}}
\put(170,90){\line(0,-1){20}}

\put(180,75){$0$}
\put(200,75){$2$}
\put(220,75){$1$}
\put(240,75){$\underline 0$}
\put(260,75){$\bar 2$}

\put(170,50){\line(1,0){40}}

\put(210,70){\line(0,-1){20}}
\put(190,70){\line(0,-1){20}}
\put(170,70){\line(0,-1){20}}

\put(180,55){$1$}
\put(200,55){$0$}
\put(220,55){$\bar 2$}

\put(170,50){\line(1,0){40}}
\put(170,30){\line(1,0){40}}

\put(210,50){\line(0,-1){20}}
\put(190,50){\line(0,-1){20}}
\put(170,50){\line(0,-1){20}}

\put(180,35){$2$}
\put(200,35){$\underline 1$}
\put(180,15){$\bar 0$}

\thicklines

\put(-10,110){\vector(0,-1){100}}
\put(-10,110){\vector(1,0){160}}

\put(170,110){\vector(0,-1){100}}
\put(170,110){\vector(1,0){160}}

\put (230, 10){$\F^{(2)}(u)$}

\end{picture}
\caption{A colored partition $(7,4,2,2)$ in Fock representations. The colors of
concave corners have bars, and  the colors of convex corners are underlined.}
\end{figure}
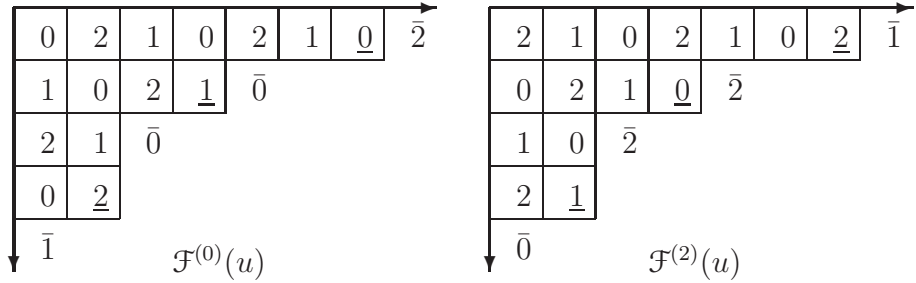

The Fock spaces were studied before using bosonizations, see for example \cite{TU}, hence the name.

It is convenient to rewrite formula \Ref{K act} in terms of the geometry of partition $\la$ as follows. Let $\la$ be a partition. The dual partition $\la'$ is given by $\la'_i=|\{j\ |\ \la_j\geq i\}|$.
A pair of natural numbers $(x,y)$ is a {\it convex corner of $\la$} if $\la'_{y+1}<\la'_y=x$.
 A pair of natural numbers $(x,y)$ is a {\it concave corner of $\la$} if $\la'_y=x-1$ and in addition
 $y=1$ or $\la'_{y-1}>x-1$, see Figure 2.
 Let $CC(\la)$ and $CV(\la)$ be the set of concave and convex corners of $\la$ respectively.
 
 For $i\in\hat I$, let 
\be
CC_i^{(k)}(\la)=\{(x,y)\in CC(\la), k+x-y\equiv i\}
\en 
be the set of concave corners of $\la$ of color $i$. Similarly, for $i\in\hat I$, let 
\be
CV_i^{(k)}(\la)=\{(x,y)\in
CV(\la), k+x-y\equiv i\}
\en
be the set of convex corners of $\la$ of color $i$.

 \begin{lem}\label{K act 2} In $\F^{(k)}(u)$ for $i\in\hat{I}$ we have
\be
K_i(z)\ket{\la}
\hspace{-10pt}\prod_{(x,y)\in\ CV_i^{(k)}(\la)}\hspace{-6pt}\psi(q_3^xq_1^yq_2u/z)
\hspace{-10pt}\prod_{(x,y)\in\ CC_i^{(k)}(\la)} \hspace{-6pt}\psi(q_3^xq_1^yq_2^2u/z)^{-1}.
\ket{\la}.
\en
In particular, 
\be
K_i\ket{\la}=q^{|CV_i^{(k)}|- |CC_i^{(k)}|}\ket{\la}. 
\en
 \end{lem}
 \begin{proof} 
 Note that if $\la_r=\la_{r+1}$ and $\la_r+i\equiv r+k$ then the two non-trivial factors  corresponding to $s=r$ and $s=r+1$ in \Ref{K act} cancel. Therefore the lemma follows from formula \Ref{K act}.
 \end{proof}
 
\section{Macmahon modules}\label{Macmahon sec}
\subsection{Tensor products of Fock modules}
Consider $\F^{(k)}(u)\otimes \F^{(l)}(v)$.
The basis of 
$\F^{(k)}(u)\otimes \F^{(l)}(v)$ is given by $\ket{\la}\otimes \ket{\mu}$, where $\la,\mu$ are partitions. We use the comultiplication to define the $\E_n$-module structure.

\begin{lem} \label{tensor two fock} We have the following cases.
\begin{enumerate}
\item The module is not well-defined if and only if 
\bea\label{resonance}
uq_2^{b+1}q_1^{b-a}=v, \qquad a,b\in\Z, \ b-a\equiv k-l. 
\ena
\item If the module is well defined it is an
irreducible quasi-finite lowest weight $\E_n$-module of level $q^2$.
\item If \Ref{resonance} holds and $a\geq 0$, $b\geq 0$ then 
there exists a well-defined submodule with basis given by vectors $\ket{\la}\otimes \ket{\mu}$  such that
\bea\label{relation}
\la_i\geq \mu_{i+b}-a,\qquad i=1,2,\dots .
\ena
It is an irreducible quasi-finite, 
tame, lowest weight $\E_n$-module of level $q^2$. 
\end{enumerate}
\end{lem}
\begin{proof}
The module $\F^{(k)}(u)\otimes \F^{(l)}(v)$ is a subspace of the tensor product
\be
 (V^{(k)}(u)\otimes V^{(k)}(uq_2^{-1})\otimes V^{(k)}(uq_2^{-2})\otimes \dots)\otimes( V^{(l)}(v)\otimes V^{(l)}(vq_2^{-1})\otimes V^{(l)}(vq_2^{-2})\otimes \dots).
\en 
Clearly, the module is well-defined, quasi-finite, irreducible 
unless there is interaction between $V^{(k)}(uq_2^{-i+1})$ and $V^{(l)}(vq_2^{-j+1})$ 
for some $i,j$. 

Due to Lemma \ref{tensor two}  $V^{(k)}(uq_2^{-i+1})\otimes V^{(l)}(vq_2^{-j+1})$ is irreducible unless condition \Ref{resonance} is satisfied and $b+i-j\in\{0,-1,-2\}$. 

In the case of \Ref{resonance} and $b+i-j=0$ this module has the submodule spanned by vectors $[uq_2^{-i+1}]^{(k)}_s\otimes[vq_2^{-j+1}]^{(l)}_t$ with $s\geq t-a+b$ which is equivalent to \Ref{relation}. 

In the case of \Ref{resonance} and $b+i-j=-2$ this module has the submodule spanned by vectors $[uq_2^{-i+1}]^{(k)}_s\otimes[vq_2^{-j+1}]^{(l)}_t$ with $s\geq t-a+b+1$. This translates to 
$\la_i\geq \mu_{b+i+2}-a-1$ which is automatic if \Ref{relation} is satisfied.

In the case of \Ref{resonance} and $b+i-j=-1$ we have a pole which may occur 
in the action of $E_r(z)$ on $[uq_2^{-i+1}]^{(k)}_s\otimes[vq_2^{-j+1}]^{(l)}_t$ with $s= t-a+b+1$ which means $\la_i=\mu_{b+i+1}-a$. Such poles appear in $\F^{(k)}(u)\otimes \F^{(l)}(v)$, but not in the submodule. Indeed, $\la_i=\mu_{b+i+1}-a$ together with  \Ref{relation} implies $\mu_{b+i+1}=\mu_{b+i}$ and therefore the corresponding matrix coefficient 
$E_r(z)$ is identically zero.
\end{proof}
We note that if \Ref{resonance} holds and $a\leq 0$, $b\leq 0$ 
we can use Lemma \ref{delta fac} to define the irreducible quotient module containing
the vector $\ket{\emptyset}\otimes \ket{\emptyset}$. 
We do not use this since the highest weights in 
$\F^{(k)}(u)\otimes \F^{(l)}(v)$ and $\F^{(l)}(v)\otimes \F^{(k)}(u)$ are the same and therefore we do not obtain any modules in addition to Lemma \ref{tensor two fock}.

However if \Ref{resonance} and $ab<0$, then we can not use the tensor product
$\F^{(k)}(u)\otimes \F^{(l)}(v)$ to describe the lowest weight irreducible 
module with lowest weight 
$((\psi(z/u))^{\delta_{ik}} (\psi(z/v))^{\delta_{il}})_{i\in\hat I}$.
One can do it using analytic continuation methods, but we do not discuss this procedure 
in this paper.
\medskip

Let $\al,\beta$ be partitions with $m$ parts such that $\al_m=\beta_m=0$. 
For $i=1,\dots, m-1$  define non-negative integers $a_i,b_i$ by
\bea\label{ab}
a_i=\al_i-\al_{i+1},\qquad b_i=\beta_i-\beta_{i+1}.
\ena
Let $u$ be a complex parameter. Define complex number $u_i$, $i=1,\dots,m$, by 
\bea\label{u}
u_1=uq_1^{\al_1}q_3^{\beta_1},\quad u_{i+1}=u_iq_2^{b_i+1}q_1^{b_i-a_i}, \qquad i=1,\dots, m-1.
\ena
Let $k\in\hat{I}$. Define $k_i\in\hat{I}$, $i=1,\dots, m$, by
\bea\label{k}
k_1=k-\al_1+\beta_1, \quad k_{i+1}\equiv k_i+a_i-b_i,\qquad i=1,\dots, m-1.
\ena
Note that $k_m=k$ and $u_m=uq_2^{m-1}$.

Consider the tensor product of vector spaces 
$\F^{(k_1)}(u_1)\otimes\dots\otimes \F^{(k_m)}(u_m)$.
Let $\N_{\al,\beta}^{(k)}(u)$ be the subspace spanned by vectors
$\ket{\la^{(1)}}\otimes\dots\otimes\ket{\la^{(m)}}$ where $\la^{(i)}$ are partitions satisfying
\bea\label{relations}
\la^{(i)}_j\geq \la^{(i+1)}_{j+b_i}-a_i, \qquad i=1,\dots,m-1.
\ena

\begin{prop}\label{tensor m Fock}
The space $\N_{\al,\beta}^{(k)}(u)$ is a well-defined submodule of the tensor product $\F^{(k_1)}(u_1)\otimes\dots\otimes \F^{(k_m)}(u_m)$. Moreover, $\N_{\al,\beta}^{(k)}(u)$ is an irreducible, tame, quasi-finite, lowest weight 
$\E_n$-module of level $q^m$.
\end{prop}
\begin{proof}
Note that the non-adjacent conditions follow from the adjacent ones given in \Ref{relations}.
\end{proof}

\subsection{The vacuum Macmahon module}
We construct the Macmahon representation by the inductive procedure 
from the Fock modules.

Fix $k\in\hat I$.
Consider the following tensor product of $r$ Fock modules:
\be
\F^{(k)}(u)\otimes \F^{(k)}(uq_2)\otimes \F^{(k)}(uq_2^2)\otimes \dots \otimes \F^{(k)}(uq_2^{r-1}).
\en
By Lemma \ref{tensor two fock} this tensor product has a submodule $\W^{(k)}_r(u)$ spanned by vectors
\be
\ket{\la^{(1)}}\otimes \ket{\la^{(2)}}\otimes \dots \otimes
\ket{\la^{(r)}}, 
\en
where $\la^{(i)}$ are arbitrary partitions satisfying 
$\la^{(i)}_j\geq\la^{(i+1)}_j$ for all $i,j$. 
Next we argue that as $r\to\infty$ one can define an appropriate limit of $\W^{(k)}_r$.

\medskip

Let $\M^{(k)}(u)$ be the subset of the infinite tensor product
\be
\M^{(k)}(u)\subset \F^{(k)}(u)\otimes \F^{(k)}(uq_2)\otimes \F^{(k)}(uq_2^{2})\otimes \dots,
\en
spanned by vectors
\be
\ket{\bs\la}=\ket{\la^{(1)}}\otimes \ket{\la^{(2)}}\otimes \dots \otimes \dots,
\en
where $\la^{(i)}$ are arbitrary partitions with only finitely many of them being nonempty
satisfying $\la^{(i)}_j\geq\la^{(i+1)}_j$. 

Let $J_r: \W^{(k)}_r(u)\to \M^{(k)}(u)$ be the injective linear map sending 
$\ket{\la^{(1)}}\otimes \dots \otimes
\ket{\la^{(r)}}$ to the vector $\ket{\la^{(1)}}\otimes \dots \otimes
\ket{\la^{(r)}}\otimes\emptyset\otimes\emptyset \dots$\ . 

Therefore we have
\be
\W^{(k)}_1(u)\subset \W^{(k)}_2 (u)\subset \W^{(k)}_3 (u)\subset \dots \subset \M^{(k)}(u).
\en

Suppose  $\bs\la$ has the property $\la^{(r-1)}=\emptyset$ for some $r\in\Z_{\geq 0}$. 
Define an action of operators $E_i(z),F_i(z),K_i(z)$ with $i\in\hat I$, $i\neq k$, and of $E_k(z)$
on $\ket{\bs\la}$ by using the action of  $\W_s^{(k)}(u)$. For example,
\be
E_i(z)\ket{\bs\la}=J_r E_i(z) (\ket{\la^{(1)}}\otimes \ket{\la^{(2)}}\otimes \dots \otimes
\ket{\la^{(r)}}).
\en

Observe that the definition does not depend on the choice of $r$. However, to make action of $F_k(z)$ and $K_k(z)$ independent on $r$ we have to modify the action. Chose $K\in\C^\times$ and set
\begin{align*}
K_k(z)\ket{\bs\la}=f_r(z)
J_r K_k(z) (\ket{\la^{(1)}}\otimes \ket{\la^{(2)}}\otimes \dots \otimes
\ket{\la^{(r)}}),\\
F_k(z)\ket{\bs\la}=f_r(z)J_r F_k(z) (\ket{\la^{(1)}}\otimes \ket{\la^{(2)}}\otimes \dots \otimes
\ket{\la^{(r)}}),
\end{align*}
where $f_r(z)=\frac{K^{-1}-Ku/z}{q^{-r}-q^ru/z}$.
Now it is easy to see that the action of $F_k(z)$ and $K_k(z)$ is also independent on the choice of $r$. 

We call $K$ {\it generic} if $K\not\in q_1^\Z q_2^\Z$.

\begin{thm}\label{M thm}  
The formulas above define an action of $\E_n$ on $\M^{(k)}(u)$.
For generic $K$, it is an irreducible, tame, lowest weight, $\Z^n$-graded module of level $K$
and lowest weight
$\bs \phi(z)=((\frac{K^{-1}-K u/z}{1-u/z})^{\delta_{ik}} )_{i\in\hat I}$.
\end{thm}
\begin{proof}
Clearly, from the construction, operators $E_i(z),F_i(z),K_i(z)$ with $i\neq k$ satisfy the same relations 
in $\M^{k}(u)$ as in $\W_r^{(k)}(u)$. The only relation which is not immediate is the commutator
$[E_k(z),F_k(z)]$.
Because of our modification it is multiplied by the function 
$f_r(z)$.

Note that $f_r(z)$ is a rational function which is well defined at zero, infinity, and $f_r(0)f_r(\infty)=1$ and it has the following property. For $v_r=\ket{\la^{(1)}}\otimes \dots \otimes
\ket{\la^{(r)}}\in\W_r^{(k)}(u)$ let $g_r(z)$
be the eigenvalue of $K_k(z)$: $K_k(z)v_r=g_r(z)v_r$.
Then for any
$r$ such that $\la^{(r)}=\emptyset$ the poles of $f_r(z)g_r(z)$ and of $g_r(z)$
are the same.
In other words the denominator $q^{-r}-q^{r}u/z$
of $f_r(z)$ is canceled with the numerator of $g_r(z)$.
 
It follows that 
\be
\delta(z/w)((g_r(z)f_r(z))^+-(g_r(z)f_r(z))^-))=f_r(z)\delta(z/w)((g_r(z))^+-(g_r(z))^-),
\en
where the suffix $+$ (respectively $-$) denotes the expansion of the rational function around $0$ (respectively $\infty$). 
Therefore the relation for the commutator of $E_k(z)$ and $F_k(z)$ is also satisfied.

The other statements of the theorem are straightforward.
\end{proof}
We denote the module described in Theorem \ref{M thm} by $\M^{(k)}(u;K)$ 
and call it the {\it vacuum Macmahon module}.

According to Theorem \ref{quasifin}, the lowest weight of 
any quasi-finite module is a product of lowest weights of $\M^{(k_i)}(u_i;K_i)$ with appropriately chosen $k_i,K_i,u_i$. 
Therefore, one can expect that all quasi-finite modules are
subquotients in a tensor product of vacuum Macmahon modules. 
Such tensor products in general have a rather complicated structure and sometimes it is more convenient 
to use an alternative construction. Therefore we generalize the construction of the vacuum Macmahon module in Section \ref{gen MM}.
We study some tensor products of Macmahon modules in Section \ref{tensor MM}.

\subsection{Colorless partitions}
In this section we define and characterize colorless partitions.

Consider the root lattice of $\sln_n$. It is a free abelian group with generators given by simple roots. 
We denote the generators by $c_i$, $i\in I$. We set
\begin{align*}
c_0=-(c_1+\cdots+c_{n-1}).
\end{align*}
We use the cyclic modulo $n$ convention for $c_i$.
For partition $\ga$, let $Y(\ga)=\{(x,y)\in\Z_{\geq 1}^2\ |\ \ga_x\geq y\}$ be the set of boxes in the
Young diagram of $\ga$.

\medskip

We call a partition $\ga$ {\it $n$-colorless} if and only if $\sum\limits_{(x,y)\in Y(\ga)}c_{x-y}=0$. 

\medskip

In other words a partition is $n$-colorless if and only if the number of boxes of color 
$i$ is the same for all $i\in\hat I$ in the coloring associated with $\F^{(k)}(u)$ for any $k$.
 For example, the partitions in Figure 2 are 3-colorless, all partitions with parts divisible by $n$  
 are $n$-colorless, the partition $(3,2,1)$ is 3-colorless but not 6-colorless.
 The number of boxes in an $n$-colorless partition is divisible by $n$. The converse is not true.
 A partition $\ga$ is colorless if and only if the dual partition $\ga'$ is colorless.
  
In this subsection, we identify $\hat I$ with the set $\{1,2,\dots,n\}$.

Define $v^{(k)}(\ga)\in\Z^n$ by the formula
\be
v^{(k)}(\ga)_i=|CC^{(k)}_i(\ga)|-|CV^{(k)}_i(\ga)|,\qquad i\in\hat I.
\en
Let $\delta^{(n)}_{i,j}=1$ for $i\equiv j$ and $\delta^{(n)}_{i,j}=0$ for $i\not\equiv j$.
The following is straightforward.
\begin{lem}\label{CCCV}
Suppose that a partition $\mu$ is obtained from a partition $\nu$
by adding a box of color $i$. Then we have
\begin{align*}
v^{(k)}(\mu)-v^{(k)}(\nu)=-\overline\al_i,
\end{align*}
where $(\overline\al_i)_j=-\delta^{(n)}_{i,j+1}-\delta^{(n)}_{i,j-1}+2\delta^{(n)}_{i,j}$.
 \qquad $\Box$
\end{lem}
Here is a characterization of colorless partitions which we will use.
\begin{prop}\label{colorless}
A partition $\ga$ is colorless if and only if for some $k\in\hat I$, we have $v^{(k)}(\ga)\in\Z_{\geq 0}^n$.
If partition $\ga$ is colorless then  $v^{(k)}(\ga)=(\delta_{ik})_{i\in\hat I}$ for all $k\in\hat I$.
\end{prop}
\begin{proof} Without loss of generality, assume $k=0$.

Let $C=(a_{ij})_{i,j\in\hat I}$ be the $\widehat{\mathfrak{sl}}_n$ Cartan matrix.
Let $c=v^{(0)}(\emptyset)=(\delta_{in})_{i\in\hat I}\in\Z_{\geq 0}^n$.

By Lemma \ref{CCCV}, it is sufficient to show that if for $b\in\Z^n$ we have
\bea\label{geq}
(c-Cb)\in\Z^n_{\geq 0},
\ena
then $b_i=a$, $i\in\hat I$, for some $a\in\Z$.

Assume \Ref{geq} and set $a=b_0$.

Multiplying \Ref{geq} by $(0,1,2,\ldots, n-1)$ from the left we obtain
\be
-nb_1+nb_n\leq n-1.
\en
Multiplying \Ref{geq} by $(n-1,n-2,\ldots,1,0)$ from the left we obtain
\begin{align*}
nb_1-nb_n\leq0.
\end{align*}
This implies $b_1=b_n=a$. By symmetry we have $b_{n-1}=a$.

Using $b_1=b_{n-1}=b_n=a$, the condition \Ref{geq} reduces to
\begin{align*}
\begin{pmatrix}
-1&&&&&&&1\\
2&-1&&&&&&-1\\
-1&2&-1&&&&&0\\
0&-1&2&-1&&&&0\\
&&\ddots&\ddots&\ddots&&&\vdots\\
&&&&-1&2&-1&0\\
&&&&&-1&2&-1\\
&&&&&&-1&1
\end{pmatrix}
\begin{pmatrix}
b_2\\
b_3\\
b_4\\
b_5\\
\vdots\\
b_{n-3}\\
b_{n-2}\\
a
\end{pmatrix}
\in \Z_{\leq 0}^{n-1},
\end{align*}
where the size of the matrix in the left hand side is $(n-1)\times(n-2)$.

The first line implies $-b_2+a\leq0$. Multiplying by $(n-2,\ldots,1,0)$ from the left
we obtain $b_2-a\leq0$. Therefore, $b_2=a$, and by symmetry $b_{n-2}=a$.
Continuing similarly, we have $b_1=b_2=\cdots=b_n=a$.
\end{proof}

 We continue our study of combinatorics of colorless partitions in Section \ref{colorless and roots}.

\subsection{The Macmahon modules $\M^{(k)}_\ga(u;K)$}\label{gen MM}
In this subsection
we generalize the construction of the previous subsection.

Let $\ga$ be a partition. Let $\W^{(k)}_{r,\ga}(u)\subset \W^{(k)}_r(u)$ be the subspace spanned by vectors
$
\ket{\la^{(1)}}\otimes \ket{\la^{(2)}}\otimes \dots \otimes
\ket{\la^{(r)}}, 
$
with $\la^{(r)}_j\geq \ga_j$.  Let $\M^{(k)}_\ga(u)$ be the subset of the infinite tensor product
\be
\M^{(k)}_\ga(u)\subset \F^{(k)}(u)\otimes \F^{(k)}(uq_2)\otimes \F^{(k)}(uq_2^{2})\otimes \dots,
\en
spanned by vectors $\ket{\bs \la}$ with the property $\la^{(i)}_j=\ga_j$ for all sufficiently large $i$.

Let $J_{r,\ga}: \W^{(k)}_{r,\ga}(u)\to \M^{(k)}_{\ga}(u)$ be the injective linear map sending the vector
$\ket{\la^{(1)}}\otimes \dots \otimes
\ket{\la^{(r)}}$ to the vector $\ket{\la^{(1)}}\otimes \dots \otimes
\ket{\la^{(r)}}\otimes\ket{\ga}\otimes\ket{\ga} \dots$\ . 

Therefore we have
\be
\W^{(k)}_{1,\ga}(u)\subset \W^{(k)}_{2,\ga} (u)\subset \W^{(k)}_{3,\ga} (u)\subset \dots \subset \M^{(k)}_{\ga}(u).
\en

Suppose  $\bs\la$ has the property $\la^{(r-1)}=\ga$ for some $r\in\Z_{\geq 0}$. 
Define the action of operators $E_i(z),F_i(z),K_i(z)$ with $i\in\hat I$ by
\begin{align}\label{renorm}
E_i(z)\ket{\bs\la}&=J_r E_i(z) (\ket{\la^{(1)}}\otimes \ket{\la^{(2)}}\otimes \dots \otimes
\ket{\la^{(r)}}),\notag\\
F_i(z)\ket{\bs\la}&=f(i,r,z)J_r F_i(z) (\ket{\la^{(1)}}\otimes \ket{\la^{(2)}}\otimes \dots \otimes
\ket{\la^{(r)}}), \\
K_i(z)\ket{\bs\la}&=f(i,r,z)J_r K_i(z) (\ket{\la^{(1)}}\otimes \ket{\la^{(2)}}\otimes \dots \otimes
\ket{\la^{(r)}}),\notag
\end{align}
where $f(i,r,z)$ are some scalar functions which we choose from the requirement of the stabilization with respect to $r$.

\begin{lem}\label{ACTION F}
The action of operators $E_i(z)$ does not depend on the choice of $r$.
The action of operators $K_i(z)$
does not depend on the choice of $r$ if and only if 
\be
f(i,r,z)=f(i,z)
\frac{\prod\limits_{(x,y)\in\ CV_i^{(k)}(\ga)}\hspace{-6pt}(1-q_2^{r-x}q_1^{y-x}u/z)}
{\prod\limits_{(x,y)\in\ CC_i^{(k)}(\ga)} \hspace{-6pt}(1-q_2^{r-x-1}q_1^{y-x-1}u/z)},
\en
where $f(i,z)$ does not depend on $r$.
\end{lem}
\begin{proof}
The lemma is obtained from Lemma \ref{K act 2} by a direct computation.
\end{proof}
It follows that in order to preserve the commutators of $E_i(z)$ and $F_i(z)$ we need $f(i,r,0)$ to be well-defined. It follows that we have to require $|CC_i^{(k)}(\ga)|\geq|CV_i^{(k)}(\ga)|$. Indeed, if not then the function $f(i,z)$ has poles. Such poles independent on $r$ cannot be canceled by
eigenvalues
 of $K_i(z)$ acting on $\ket{\la^{(r)}}=\ket{\ga}$, and therefore our modification
 \Ref{renorm} does not preserve the relation in $\E_n$ for $[E_i(z),F_i(w)]$. 

By Proposition \ref{colorless} 
the inequalities $|CC_i^{(k)}(\ga)|\geq|CV_i^{(k)}(\ga)|$ hold for all $i\in \hat I$ if and only if $\ga$ is a colorless partition. Moreover, if $\ga$ is a colorless partition, then $|CC_i^{(k)}(\ga)|-|CV_i^{(k)}(\ga)|=\delta_{ik}$. 

Let $\ga$ be a colorless partition.
For $i\in\hat I$, $i\neq k$, let $f(i,z)=f(i)$ be uniquely determined by the condition $f(i,r,0)f(i,r,\infty)=1$.
Let
\begin{align}
f(k,z)=(K^{-1}-Ku/z)f(k),\label{FK}
\end{align}
where $f(k)$ is uniquely determined by the condition $f(k,r,0)f(k,r,\infty)=1$.

\begin{thm}\label{M2 thm}  For any colorless partition $\ga$ and $K\in\C^\times$,
the formulas above define an action of $\E_n$ on $\M^{(k)}_\ga(u)$.
For generic $K$, it is an irreducible, tame, lowest weight, $\Z^n$-graded module of level $K$ and lowest weight 
\be
\left(\frac
{\prod\limits_{(x,y)\in\ CV_i^{(k)}(\ga)}\hspace{-6pt}(q^{x}q_1^{(x-y)/2} -q^{-x}q_1^{(y-x)/2}u/z)}
{\prod\limits_{(x,y)\in\ CC_i^{(k)}(\ga)} \hspace{-6pt}( q^{x-1}q_1^{(x-y)/2}-q^{1-x}q_1^{(y-x)/2}u/z)} 
\ (K^{-1}-Ku/z)^{\delta_{ik}}\right)_{i\in\hat I}.
\en
\end{thm}
\begin{proof} The proof is similar to the proof of Theorem \ref{M thm}. 
The new phenomena is vanishing of extra terms in $F_i(z)$. For $v=\ket{\la^{(1)}}\otimes \dots \otimes
\ket{\la^{(r)}}\in\W_{r,\ga}^{(k)}(u)$ with $\la^{(r)}=\ga$, the result of action of $F_i(z)$ on $v$, $F_i(z)v$, contains terms
$\ket{\la^{(1)}}\otimes \dots \otimes F_i(z)\ket{\la^{(r)}}$ which do not appear when $r$ is replaced by $r+1$.
However, it is straightforward to verify that $f(i,r,z)F_i(z)\ket{\la^{(r)}}=0$. 
In other words, a zero of $f(i,r,z)$ coincides with
the argument of delta function of $F_i(z)$ computed on
$\ket{\ga}$ in $\F^{(k)}(uq_2^{r-1})$. It follows that
$F_i(z)\ket{\bs\la}$ does not depend on the choice of $r$.

We leave the rest of the detail to the reader.
\end{proof}
We denote $\M^{(k)}_\ga(u;K)$ the module described in Theorem \ref{M2 thm}.

\subsection{The general Macmahon module}
In this section we describe the most general Macmahon module depending on three partitions.
Let $\al,\beta,\ga$ be three partitions such that $\al_m=\beta_m=0$ and let $k\in\hat I$,  $K\in\C^\times$.

Consider the tensor product
\begin{align}
\F^{(k_1)}(u_1)\otimes\F^{(k_2)}(u_2)\otimes\dots\otimes\F^{(k_m)}(u_m)\otimes
\M^{(k)}_\ga(q_2u_m;Kq^{-m}),\label{WHAT IS M}
\end{align}
where $k_i, u_i$ are defined by \Ref{ab}-\Ref{k}.

Set $a_i=b_i=0$ for $i\geq m$ and let $\mathcal M_{\al,\beta,\ga}^{(k)}(u;K)$
be the subspace of the tensor product spanned by vectors
$\ket{\la^{(1)}}\otimes\dots,\otimes\ket{\la^{(m)}}\otimes\ket{\la^{m+1}}\otimes\dots,$ where $\la_j^{(i)}\geq\la_{j+b_i}^{(i+1)}-a_i$, $i\in\Z_{\geq 1}$ and $\la^{(i)}=\ga$ for all sufficiently large $i$. Here
$\ket{\la^{(i)}}\in\F^{(k_i)}(u_i)$, $i=1,\dots,m$,
$\ket{\la^{m+1}}\otimes\dots \in \mathcal M^{(k)}_\ga(q_2u_m;Kq^{-m})$. 

For given partition $\la$ and nonnegative integers $a, b$ define the partition $\la[a,b]$ by the rule
$\la[a,b]_i=\max\{\la_{i+b}-a,0\}$. Define $\ga^{(i)}$, $i\in\Z_{\geq 1}$
inductively by setting 
\be
\ga^{(i)}=\ga \quad (i>m), \qquad \ga^{(i)}=\ga^{(i+1)}[a_i,b_i] \quad (i=1,\dots,m).
\en
Set 
\be
v_{\al,\beta,\ga}^{(k)}=\ket{\ga^{(1)}}\otimes\ket{\ga^{(2)}}\otimes \dots \in\M_{\al,\beta,\ga}^{(k)}(u;K).
\en

\begin{thm}\label{gen Mm}
The space $\M_{\al,\beta,\ga}^{(k)}(u;K)$ is a well-defined $\E_n$-submodule in the tensor product $\F^{(k_1)}(u_1)\otimes\F^{(k_2)}(u_2)\otimes\dots\otimes\F^{(k_m)}(u_m)\otimes
\M^{(k)}_\ga(q_2u_m;Kq^{-m})$. Moreover, for generic $K$, it is an irreducible, tame, quasi-finite, lowest weight, $\Z^n$-graded $\E_n$-module of level $K$
with lowest weight vector $v_{\al,\beta,\ga}^{(k)}$.
\end{thm}
\begin{proof}
Theorem follows from Lemma \ref{tensor two fock}, cf. also Proposition \ref{tensor m Fock}.
\end{proof}
We call the $\E_n$-module $\M_{\al,\beta,\ga}^{(k)}(u;K)$ 
constructed in Theorem \ref{gen Mm} the {\it Macmahon module}.
 
We remark that, equivalently, one could construct the Macmahon modules inside the tensor product $\N_{\al,\beta}^{(k)}(u)\otimes\M_\ga^{(k)}(q_2u_m;Kq^{-m})$.

It is convenient to visualize the basis of $\M_{\al,\beta,\ga}^{(k)}(u;K)$  by using plane partitions with asymptotic conditions. 

A plane partition $\La$ with asymptotic conditions $\al,\beta,\ga$ is a set
$\{\La_{x,y}\}_{x,y\in\Z_{\geq 1}}$, where $\La_{x,y}=\infty$ for $y\leq\ga_x$ and $\La_{x,y}\in\Z_{\geq 0}$ otherwise, satisfying
$\La_{x,y}\geq \La_{x+1,y}$, $\La_{x,y}\geq \La_{x,y+1}$,
$\La_{x,y}=\al'_y$
for sufficiently large $x$ and
$\La_{x,y}=\beta'_x$
for sufficiently large $y$.
A triple $(x,y,z)\in\Z_{\geq 1}^3$ is called a box in the plane partition ${\La}$ if and only if
\begin{align*}
z\leq{\La}_{x,y}.
\end{align*}
A plane partition can be visualized by a set of boxes (see Figure 3).

The vector $\ket{\bs \la}\in \M_{\al,\beta,\ga}^{(k)}(u;K)$ is identified with the plane partition $\La(\bs\la)$ given by
\be
\La(\bs\la)_{x,y}=|\{s\in\Z_{\geq1}\ |\ y\leq\al_s\}\cup
\{s\in\Z_{\geq1}\ |\ x\leq\beta_s\}\cup
\{s\in\Z_{\geq1}\ |\ \la^{(s)}_{x-\beta_s}\geq y-\al_s\}|.
\en
It is clear that $\La(\bs\la)$ is a plane partition with asymptotic conditions $\al,\beta,\ga$.

When we discuss the module $\M_{\al,\beta,\ga}^{(k)}(u;K)$
we define color and evaluation parameter of each box in such a way that
the box $(1,1,1)$ has color $k$ and evaluation parameter $u$. In general, 
the box $(x,y,z)$ has color $k+x-y$ and evaluation parameter given by
\begin{align*}
u_{x,y,z}=uq_3^{x}q_1^{y}q_2^z.
\end{align*}
Then, again, we say that the operators $E_i(w)$ and $F_i(w)$
add and remove a box $(x,y,z)$ of color $i$
and the coefficient of the action contains the delta function $\delta(u_{x,y,z}/w)$.

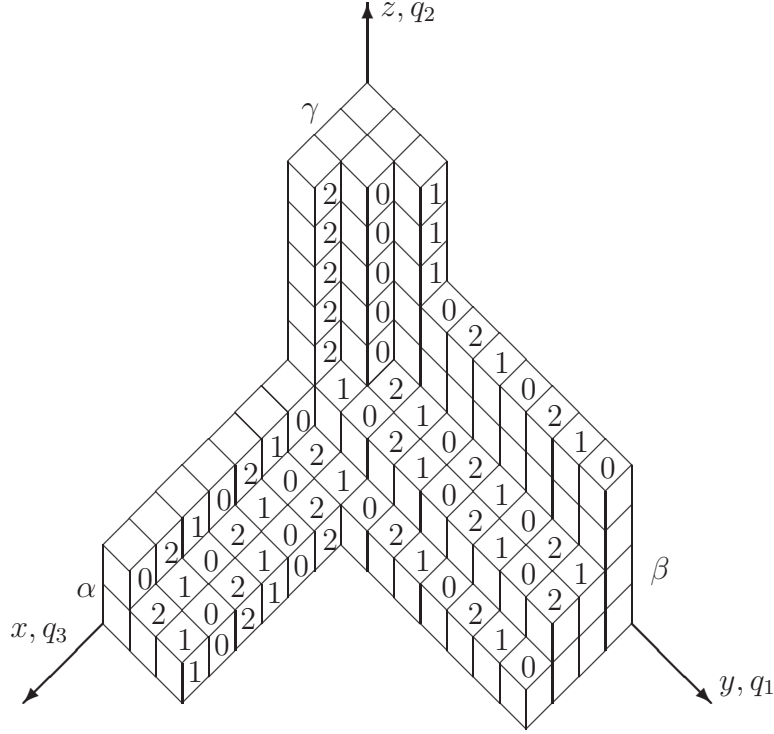
\begin{figure}
\begin{picture}(30,290)(200,-20)

\put(200,235){\line(1,-1){30}}
\put(190,225){\line(1,-1){20}}
\put(180,215){\line(1,-1){10}}
\put(170,205){\line(1,-1){10}}

\put(200,235){\line(-1,-1){30}}
\put(210,225){\line(-1,-1){20}}
\put(220,215){\line(-1,-1){10}}

\put(175,222){$\gamma$}
\put(205,260){$z,q_2$}


\put(230,205){\line(-1,-1){10}}
\put(180,195){\line(1,1){10}}
\put(180,180){\line(1,1){10}}
\put(180,165){\line(1,1){10}}
\put(180,150){\line(1,1){10}}
\put(180,135){\line(1,1){10}}

\put(183,189){$2$}
\put(183,174){$2$}
\put(183,159){$2$}
\put(183,144){$2$}
\put(183,129){$2$}

\put(220,195){\line(-1,1){10}}

\put(190,205){\line(1,-1){10}}
\put(190,190){\line(1,-1){10}}
\put(190,175){\line(1,-1){10}}
\put(190,160){\line(1,-1){10}}
\put(190,145){\line(1,-1){10}}

\put(210,190){\line(1,-1){10}}
\put(210,175){\line(1,-1){10}}
\put(210,160){\line(1,-1){10}}
\put(210,145){\line(1,-1){80}}

\put(210,190){\line(-1,-1){10}}
\put(210,175){\line(-1,-1){10}}
\put(210,160){\line(-1,-1){10}}
\put(210,145){\line(-1,-1){10}}

\put(203,189){$0$}
\put(203,174){$0$}
\put(203,159){$0$}
\put(203,144){$0$}
\put(203,129){$0$}

\put(230,190){\line(-1,-1){10}}
\put(230,175){\line(-1,-1){10}}
\put(230,160){\line(-1,-1){10}}

\put(223,189){$1$}
\put(223,174){$1$}
\put(223,159){$1$}

\put(210,205){\line(-1,-1){10}}

\put(100,30){\line(0,1){30}}
\put(110,20){\line(0,1){30}}
\put(120,10){\line(0,1){15}}
\put(130,0){\line(0,1){15}}

\put(100,30){\line(1,-1){30}}
\put(100,60){\line(1,-1){10}}

\put(130,0){\line(0,1){15}}

\put(90,40){$\al$}


\put(130,15){\line(-1,1){30}}
\put(140,25){\line(-1,1){20}}
\put(150,35){\line(-1,1){20}}
\put(160,45){\line(-1,1){20}}
\put(170,55){\line(-1,1){20}}
\put(180,65){\line(-1,1){20}}
\put(190,75){\line(-1,1){20}}

\put(140,25){\line(0,-1){15}}
\put(150,35){\line(0,-1){15}}
\put(160,45){\line(0,-1){15}}
\put(170,55){\line(0,-1){15}}
\put(180,65){\line(0,-1){15}}

\put(300,30){\line(0,1){60}}
\put(290,20){\line(0,1){30}}
\put(280,10){\line(0,1){30}}
\put(270,0){\line(0,1){15}}

\put(300,30){\line(-1,-1){40}}
\put(300,90){\line(-1,-1){10}}
\put(300,75){\line(-1,-1){10}}
\put(300,60){\line(-1,-1){10}}
\put(300,45){\line(-1,-1){30}}

\put(307,47){$\beta$}


\put(290,100){\line(-1,-1){10}}
\put(280,110){\line(-1,-1){10}}
\put(270,120){\line(-1,-1){10}}
\put(260,130){\line(-1,-1){10}}
\put(250,140){\line(-1,-1){10}}
\put(240,150){\line(-1,-1){10}}

\put(228,145){$0$}
\put(238,135){$2$}
\put(248,125){$1$}
\put(258,115){$0$}
\put(268,105){$2$}
\put(278,95){$1$}
\put(288,85){$0$}

\put(290,80){\line(0,-1){30}}
\put(280,90){\line(0,-1){30}}
\put(270,100){\line(0,-1){30}}
\put(260,110){\line(0,-1){30}}
\put(250,120){\line(0,-1){30}}
\put(240,130){\line(0,-1){30}}
\put(230,140){\line(0,-1){30}}
\put(220,150){\line(0,-1){30}}

\put(290,50){\line(-1,-1){20}}
\put(280,60){\line(-1,-1){20}}
\put(270,70){\line(-1,-1){20}}
\put(260,80){\line(-1,-1){20}}
\put(250,90){\line(-1,-1){20}}
\put(240,100){\line(-1,-1){20}}
\put(230,110){\line(-1,-1){20}}
\put(220,120){\line(-1,-1){20}}
\put(210,130){\line(-1,-1){20}}

\put(208,115){$2$}
\put(218,105){$1$}
\put(228,95){$0$}
\put(238,85){$2$}
\put(248,75){$1$}
\put(258,65){$0$}
\put(268,55){$2$}
\put(278,45){$1$}

\put(188,115){$1$}
\put(198,105){$0$}
\put(208,95){$2$}
\put(218,85){$1$}
\put(228,75){$0$}
\put(238,65){$2$}
\put(248,55){$1$}
\put(258,45){$0$}
\put(268,35){$2$}

\put(178,90){$2$}
\put(188,80){$1$}
\put(198,70){$0$}
\put(208,60){$2$}
\put(218,50){$1$}
\put(228,40){$0$}
\put(238,30){$2$}
\put(248,20){$1$}
\put(258,10){$0$}

\put(168,80){$0$}
\put(158,70){$1$}
\put(148,60){$2$}
\put(138,50){$0$}
\put(128,40){$1$}
\put(118,30){$2$}

\put(178,70){$2$}
\put(168,60){$0$}
\put(158,50){$1$}
\put(148,40){$2$}
\put(138,30){$0$}
\put(128,20){$1$}

\put(183,58){$2$}
\put(172,48){$0$}
\put(162,38){$1$}
\put(152,28){$2$}
\put(142,18){$0$}
\put(132,8){$1$}

\put(173,103){$0$}
\put(163,93){$1$}
\put(153,83){$2$}
\put(143,73){$0$}
\put(133,63){$1$}
\put(123,53){$2$}
\put(113,43){$0$}

\put(270,30){\line(0,-1){15}}
\put(260,40){\line(0,-1){15}}
\put(250,50){\line(0,-1){15}}
\put(240,60){\line(0,-1){15}}
\put(230,70){\line(0,-1){15}}
\put(220,80){\line(0,-1){15}}
\put(210,90){\line(0,-1){15}}
\put(200,100){\line(0,-1){15}}
\put(190,110){\line(0,-1){15}}

\put(250,15){\line(0,-1){15}}
\put(240,25){\line(0,-1){15}}
\put(230,35){\line(0,-1){15}}
\put(220,45){\line(0,-1){15}}
\put(210,55){\line(0,-1){15}}
\put(200,65){\line(0,-1){15}}

\put(250,15){\line(1,1){10}}
\put(240,25){\line(1,1){10}}
\put(230,35){\line(1,1){10}}
\put(220,45){\line(1,1){10}}
\put(210,55){\line(1,1){10}}
\put(200,65){\line(1,1){10}}
\put(190,75){\line(1,1){10}}
\put(270,15){\line(-1,-1){10}}
\put(260,5){\line(0,-1){15}}

\put(170,205){\line(0,-1){75}}
\put(180,195){\line(0,-1){75}}

\put(100,60){\line(1,1){70}}
\put(110,50){\line(1,1){70}}

\put(190,205){\line(0,-1){75}}
\put(110,35){\line(1,1){70}}
\put(120,25){\line(1,1){70}}
\put(260,-10){\line(-1,1){70}}
\put(130,0){\line(1,1){60}}


\put(280,40){\line(-1,1){80}}
\put(260,5){\line(-1,1){70}}
\put(130,15){\line(1,1){60}}

\put(300,90){\line(-1,1){70}}
\put(230,205){\line(0,-1){45}}

\put(290,80){\line(-1,1){70}}
\put(220,195){\line(0,-1){45}}

\put(290,50){\line(-1,1){80}}
\put(210,205){\line(0,-1){75}}
\put(270,30){\line(-1,1){90}}
\put(200,195){\line(0,-1){75}}
\put(270,15){\line(-1,1){90}}

\put(180,105){\line(0,1){15}}
\put(180,120){\line(1,1){10}}

\put(200,120){\line(1,1){10}}
\put(200,120){\line(-1,1){10}}

\put(190,60){\line(0,1){15}}
\put(180,180){\line(-1,1){10}}
\put(180,165){\line(-1,1){10}}
\put(180,150){\line(-1,1){10}}
\put(180,135){\line(-1,1){10}}
\put(180,120){\line(-1,1){10}}
\put(170,110){\line(-1,1){10}}
\put(160,100){\line(-1,1){10}}
\put(150,90){\line(-1,1){10}}
\put(140,80){\line(-1,1){10}}
\put(130,70){\line(-1,1){10}}
\put(120,60){\line(-1,1){10}}
\put(170,110){\line(-1,1){10}}
\put(160,100){\line(-1,1){10}}

\put(170,110){\line(0,-1){15}}
\put(160,100){\line(0,-1){15}}
\put(150,90){\line(0,-1){15}}
\put(140,80){\line(0,-1){15}}
\put(130,70){\line(0,-1){15}}
\put(120,60){\line(0,-1){15}}
\put(170,110){\line(0,-1){15}}
\put(160,100){\line(0,-1){15}}

\put(65,25){$x,q_3$}
\put(333,5){$y,q_1$}

\thicklines
\put(200,235){\vector(0,1){30}}
\put(300,30){\vector(1,-1){30}}
\put(100,30){\vector(-1,-1){30}}
\end{picture}
\caption{The colored plane partition corresponding to the lowest weight vector of $\E_3$-module 
$\M_{\al,\beta,\ga}^{(0)}(u;K)$
with $\al=(3,1)$, $\beta=(4,3,1,1)$, $\ga=(3,2,1)$.}
\end{figure}
\label{Plane partition}

\subsection{Special values of $K$}\label{special-vallue}
Let $\M_{\al,\beta,\ga}^{(k)}(u;K)$ be a Macmahon module. For generic $K$ the Macmahon module is irreducible. For special values of $K$ 
the Macmahon module becomes reducible because the matrix coefficients of $F_k(w)$ contain delta functions multiplied by 
an overall factor $K^{-1}-Ku/w$.  Therefore some matrix coefficients of  $F_k(w)$
vanish for special values of $K$.

Given partitions $\al,\beta,\gamma$, we call a box $(x,y,z)$ {\it special for $\al,\beta,\gamma$} if and only if it satisfies at least one of the following three conditions:
\bea\label{rrrr}
z=\al'_x+1\geq \beta'_y+1\quad  &{\rm and}&\quad  y\geq\ga_x+1;\notag\\
z=\beta'_y+1\geq\al'_x+1 \quad &{\rm and}&\quad  y\geq\ga_x+1;\\
y=\ga_x+1\quad  {\rm or} \quad  x=\ga'_y+1, \quad &{\rm and}&\quad  z\geq \al'_x+1,\ z\geq \beta'_y+1.\notag
\ena
If $\bs\La$ is a plane partition with asymptotic conditions $\al,\beta,\ga$ and if $(x,y,z)$ is a box in $\bs \La$, then there exists a unique $t\in\Z{\geq 0}$ such that $(x-t,y-t,z-t)$ is special for $\al,\beta,\ga$.

Let $(x,y,z)$ be special for $\al,\beta,\ga$, and let $t\in\Z_{\geq 0}$.
Let  $U_{\al,\beta,\ga}^{(k),(x,y,z;t)}(u;K)$
(resp., $W_{\al,\beta,\ga}^{(k),(x,y,z;t)}(u;K)$) be the subspaces 
of $\M_{\al,\beta,\ga}^{(k)}(u;K)$ spanned by vectors
$\ket{\bs \la}$ such that the corresponding plane partition contains
(resp., does not contain) the box $(x+t,y+t,z+t)$. In addition, set $U_{\al,\beta,\ga}^{(k),(x,y,z;-1)}(u;K)=\M_{\al,\beta,\ga}^{(k)}(u;K)$,
$W_{\al,\beta,\ga}^{(k),(x,y,z;-1)}(u;K)=0$. 

We have
\begin{align*}
&\cdots\subset U_{\al,\beta,\ga}^{(k),(x,y,z;1)}(u;K)
\subset U_{\al,\beta,\ga}^{(k),(x,y,z;0)}(u;K)\subset
U_{\al,\beta,\ga}^{(k),(x,y,z;-1)}(u;K)=\M_{\al,\beta,\ga}^{(k)}(u;K),\\
&\cdots\supset W_{\al,\beta,\ga}^{(k),(x,y,z;1)}(u;K)
\supset W_{\al,\beta,\ga}^{(k),(x,y,z;0)}(u;K)\supset
W_{\al,\beta,\ga}^{(k),(x,y,z;-1)}(u;K)=0.
\end{align*}
 
For $\ket{\bs\la}\in\M_{\al,\beta,\ga}^{(k)}(u;K)$, the vector $F_k(w)\ket{\bs\la}$ is a linear combination of terms which correspond to removing a box of color $k$ from the corresponding plane partition $\La(\bs\la)$.
If a box $(x,y,z)$ has color $k$, that is if $x\equiv y$,
then the corresponding term contains
the delta function $\delta(q_3^{x}q_1^{y}q_2^{z}u/w)$. In addition every matrix coefficient 
of $F_k(w)$ (and also of $K_k(w)$) contains the factor
\begin{align}\label{FACTOR}
q^mK^{-1}-K q^{-m}q_2u_m/w=q^m(K^{-1}-Ku/w),
\end{align}
see Lemma \ref{ACTION F} and \eqref{FK}.

\begin{prop}\label{Spec K} 
Let $(x,y,z)$ be a box of color $k$ which is special for $(\al,\beta,\ga)$,
and let $K$ be given by
\bea\label{K resonance}
K^2=q_3^{x}q_1^{y}q_2^{z}.
\ena
Then, for $t\in\Z_{\geq 0}$,
the space $U_{\al,\beta,\ga}^{(k),(x,y,z;t)}(u)$ is a submodule of the Macmahon module
$\M_{\al,\beta,\ga}^{(k)}(u;K)$.
The quotient module
$\Hc_{\al,\beta,\ga}^{(k),(x,y,z;t)}(u)=U_{\al,\beta,\ga}^{(k),(x,y,z;t-1)}(u;K)/U_{\al,\beta,\ga}^{(k),(x,y,z;t)}(u;K)$
is an irreducible lowest weight, tame, quasi-finite, $\Z^n$-graded $\E_n$-module of level
$(q_3^xq_1^yq_2^z)^{1/2}$.
\end{prop}

\begin{proof}
For such non-generic $K$ the module $\M_{\al,\beta,\ga}^{(k)}(u;K)$ remains well-defined and 
cyclic but becomes reducible. Namely, the operator $F_k(w)$ cannot remove the box of color $k$ at $(x,y,z)$.
More precisely, if the plane partition of $\bs \la$ contains the box $(x,y,z)$ and the plane partition of $\bs\mu$ does not then the matrix element vanishes: $\bra{\bs\mu}F_k(w)\ket{\bs\la}=0$.
Therefore we have non-trivial submodules. The rest of
Proposition is proved by a straightforward check.
\end{proof}
Note that we drop $K$ from the notation $\Hc_{\al,\beta,\ga}^{(k),(x,y,z;t)}(u)$
as $K$ is computed by \Ref{K resonance}.

 \medskip
We note that for partitions $\al,\beta$ such that $\al_{m}=\beta_m=0$ we have
$\N^{(k)}_{\al,\beta}\simeq \Hc_{\al,\beta,0}^{(k),(1,1,m;0)}(u)$. 
The $n=1$ analogs of modules $\Hc_{\al,\beta,0}^{(k),(s+1,s+1,1;0)}(u)$ 
were studied in detail in \cite{FJMM}.
In the case when two out of three conditions \Ref{rrrr} are satisfied simultaneously
then $\Hc_{\al,\beta,0}^{(k),(x,y,z;0)}(u)$ can be written as a non-trivial tensor product
which is also discussed in Section 3.5 of \cite{FJMM} for $n=1$. 
Since the treatment of these cases is just parallel to the case of $n=1$, we leave the details to the reader.

Let $\al$, $\ga$ be partitions such that $\ga$ is $n$-colorless and $\al_1<n$, $\ga_1<n$.
Define
\begin{align}\label{DEF G}
{\mathcal G}_{\al,\ga}^{(k)}(u)= \Hc_{\al,\emptyset,\ga}^{(k),(1,n+1,1;0)}(u).
\end{align}
The basis of ${\mathcal G}_{\al,\ga}^{(k)}(u)$ is labeled by plane partitions with asymptotic conditions $(\al,\emptyset,\ga)$ which do not contain box $(1,n+1,1)$. In other words, it is labeled by $n$-tuples of partitions $\mu^{(1)},\dots,\mu^{(n)}$ such that $\mu^{(i)}_{j}\geq \mu^{(i+1)}_{j+\ga_i'-\ga_{i=1}'}-\al_i'+\al_{i+1}'$, see Figure 4.

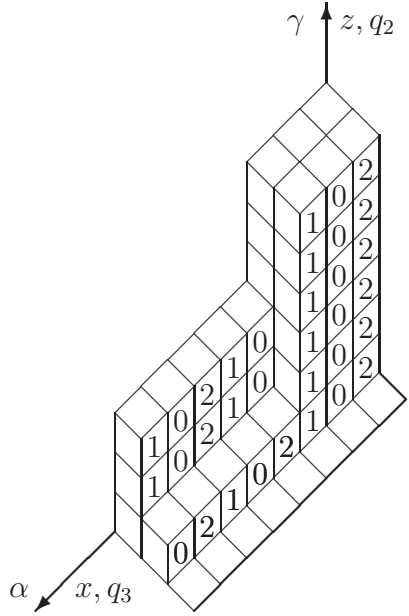
\begin{figure}
\begin{picture}(30,270)(150,-80)

\put(150,150){\line(-1,-1){30}}
\put(160,140){\line(-1,-1){30}}
\put(170,130){\line(-1,-1){30}}
\put(170,115){\line(-1,-1){30}}
\put(170,100){\line(-1,-1){30}}
\put(170,85){\line(-1,-1){30}}
\put(170,70){\line(-1,-1){30}}
\put(170,55){\line(-1,-1){30}}
\put(170,40){\line(-1,-1){80}}

\put(120,75){\line(-1,-1){50}}
\put(130,65){\line(-1,-1){50}}
\put(130,50){\line(-1,-1){50}}
\put(130,35){\line(-1,-1){50}}
\put(140,25){\line(-1,-1){50}}

\put(150,150){\line(1,-1){20}}
\put(140,140){\line(1,-1){20}}
\put(130,130){\line(1,-1){20}}
\put(120,120){\line(1,-1){20}}
\put(120,105){\line(1,-1){20}}
\put(120,90){\line(1,-1){20}}
\put(120,75){\line(1,-1){20}}

\put(130,50){\line(1,-1){10}}
\put(130,35){\line(1,-1){10}}

\put(120,25){\line(1,-1){10}}
\put(110,15){\line(1,-1){10}}
\put(100,5){\line(1,-1){10}}
\put(90,-5){\line(1,-1){10}}
\put(80,-15){\line(1,-1){10}}

\put(110,65){\line(1,-1){10}}
\put(100,55){\line(1,-1){10}}
\put(90,45){\line(1,-1){10}}
\put(80,35){\line(1,-1){10}}
\put(70,25){\line(1,-1){10}}

\put(70,-5){\line(1,-1){10}}
\put(70,10){\line(1,-1){10}}

\put(70,-20){\line(1,-1){30}}

\put(100,-30){\line(1,-1){10}}
\put(110,-20){\line(1,-1){10}}
\put(120,-10){\line(1,-1){10}}
\put(130,0){\line(1,-1){10}}
\put(140,10){\line(1,-1){10}}
\put(150,20){\line(1,-1){10}}
\put(160,30){\line(1,-1){10}}

\put(135, 170){$\gamma$}
\put(155,170){$z,q_2$}

\put(120,120){\line(0,-1){45}}
\put(130,110){\line(0,-1){75}}
\put(140,100){\line(0,-1){90}}
\put(150,110){\line(0,-1){90}}
\put(160,120){\line(0,-1){90}}
\put(170,130){\line(0,-1){90}}

\put(120,55){\line(0,-1){30}}
\put(110,45){\line(0,-1){30}}
\put(100,35){\line(0,-1){30}}
\put(90,25){\line(0,-1){30}}
\put(80,15){\line(0,-1){30}}
\put(70,25){\line(0,-1){45}}
\put(80,-15){\line(0,-1){15}}
\put(90,-25){\line(0,-1){15}}

\put(100,-15){\line(0,-1){15}}
\put(110,-5){\line(0,-1){15}}
\put(120,5){\line(0,-1){15}}
\put(130,15){\line(0,-1){15}}

\put(30,-45){$\al$}
\put(55,-45){$x,q_3$}

\put(162,113){$2$}
\put(162,98){$2$}
\put(162,83){$2$}
\put(162,68){$2$}
\put(162,53){$2$}
\put(162,38){$2$}

\put(152,103){$0$}
\put(152,88){$0$}
\put(152,73){$0$}
\put(152,58){$0$}
\put(152,43){$0$}
\put(152,28){$0$}

\put(142,93){$1$}
\put(142,78){$1$}
\put(142,63){$1$}
\put(142,48){$1$}
\put(142,33){$1$}
\put(142,18){$1$}

\put(132,8){$2$}
\put(122,-2){$0$}
\put(112,-12){$1$}
\put(102,-22){$2$}
\put(92,-32){$0$}

\put(132,8){$2$}
\put(122,-2){$0$}
\put(112,-12){$1$}
\put(102,-22){$2$}
\put(92,-32){$0$}

\put(122,33){$0$}
\put(112,23){$1$}
\put(102,13){$2$}
\put(92,3){$0$}
\put(82,-7){$1$}

\put(122,48){$0$}
\put(112,38){$1$}
\put(102,28){$2$}
\put(92,18){$0$}
\put(82,8){$1$}

\thicklines
\put(180,30){\line(-1,-1){80}}
\put(170,40){\line(1,-1){10}}
\put(150,150){\vector(0,1){30}}
\put(70,-20){\vector(-1,-1){30}}
\end{picture}
\caption{The colored plane partition corresponding to the lowest weight vector of $\E_3$-module 
$\mathcal G_{\al,\ga}^{(0)}(u)$
with $\al=(2,1,1)$, $\ga=(2,2,2)$.}
\end{figure}

The properties of modules ${\mathcal G}_{\al,\ga}^{(k)}(u)$ are different from the case $n=1$,  and we study these modules in Section \ref{G sec}.

\subsection{Tensor products of Macmahon modules}\label{tensor MM}
In this section we study tensor products of Macmahon modules.

Consider the tensor product of two Macmahon modules
\bea\label{MtM}
\M_{\al,\beta,\ga}^{(k)}(u;K_1)\otimes \M_{\mu,\eta,\nu}^{(l)}(v;K_2).
\ena
Generically, it is an irreducible module which has a basis parameterized by pairs of plane partition
$\bs\la,\bs\mu$ with corresponding asymptotic conditions.
\begin{lem}
{\rm(i) }
Assume that $(x,y,z)$ is special and of color $l$,
$K_2^2=q_3^{x}q_1^{y}q_2^{z}u/v$ and $u,v,K_1$ are generic.
Then, for $t\in\Z_{\geq 0}$, the tensor product
$U_{\al,\beta,\ga}^{(k),(x,y,z;t)}(u;K_1)\otimes \M_{\mu,\eta,\nu}^{(l)}(v;K_2)$
is a submodule of the tensor product \Ref{MtM}.\\
{\rm(ii) }
Assume that $(x,y,z)$ is special and of color $k$,
$K_1^2=q_3^{x}q_1^{y}q_2^{z}v/u$ and $u,v,K_2$ are generic.
Then, for $t\in\Z_{\geq 0}$, $\M_{\al,\beta,\ga}^{(k)}(u;K_1)\otimes W_{\mu,\eta,\nu}^{(l),(x,y,z;t)}(v;K_2)$
is a submodule of the tensor product \Ref{MtM}.
\end{lem}
\begin{proof}
The argument is similar to the proof of Proposition \ref{Spec K}.
\end{proof}

More generally,
consider the tensor product of $m$ Macmahon modules 
\bea\label{mMM}
\M^{(k_1)}_{\al^{(1)},\beta^{(1)},\gamma^{(1)}}(u_1;K_1)\otimes\dots\otimes\M^{(k_m)}_{\al^{(m)},\beta^{(m)},\gamma^{(m)}}(u_m;K_m).
\ena
Let $(x_i,y_i,z_i)\in\C^3$, $i=1,\dots,m$. Let 
\be
K_i^2=q_3^{x_{i+1}}q_1^{y_{i+1}}q_2^{z_{i+1}}u_{i+1}/u_i,
\en
$i=1,\dots,m$, where we identified indexes $m+1$ and $1$.

Let $(x,y,z)\in\C^3$. If there exists $t\in\Z_{\geq0}$ such that the box
$(x_0,y_0,z_0)=(x-t,y-t,z-t)$ is special for $\al,\beta,\ga$ and has color $l$, set
\begin{align*}
U_{\al,\beta,\ga}^{(k,l),(x,y,z)}(u;K)=U_{\al,\beta,\ga}^{(k),(x_0,y_0,z_0;t)}(u;K),\\
W_{\al,\beta,\ga}^{(k,l),(x,y,z)}(u;K)=W_{\al,\beta,\ga}^{(k),(x_0,y_0,z_0;t)}(u;K).
\end{align*}
Set
$U_{\al,\beta,\ga}^{(k,l),(x,y,z)}(u;K)=\M_{\al,\beta,\ga}^{(k)}(u;K)$  and $W_{\al,\beta,\ga}^{(k,l),(x,y,z)}(u;K)=0$
otherwise.

Set
\begin{align*}
&U_1=U^{(k_1,k_m),(x_1,y_1,z_1)}_{\al^{(1)},\beta^{(1)},\ga^{(1)}}(u_1;K_1),\
U^+_1=U^{(k_1,k_m),(x_1+1,y_1+1,z_1+1)}_{\al^{(1)},\beta^{(1)},\ga^{(1)}}(u_1;K_1),\\
&W_i=W^{(k_i,k_{i-1}),(x_i,y_i,z_i)}_{\al^{(i)},\beta^{(i)},\ga^{(i)}}(u_i;K_i),\
W^+_i=W^{(k_i,k_{i-1}),(x_i+1,y_i+1,z_i+1)}_{\al^{(i)},\beta^{(i)},\ga^{(i)}}(u_i;K_i),\ (i=2,\dots,m).
\end{align*}
We denote
\begin{align*}
\bs k=(k_1,\ldots,k_m),
(\bs x,\bs y,\bs z)=(x_1,y_1,z_1;\dots;x_m,y_m,z_m).
\end{align*}
Then we have
\begin{prop}
For generic $u_1,\dots,u_m$, the tensor product \Ref{mMM} of the Macmahon modules has a subquotient 
\begin{align*}
\lefteqn{\Hc^{\bs k,(\bs x,\bs y,\bs z)}_{\bs\al,\bs\beta,\bs \ga}(\bs u)= 
(U_1\otimes W^+_2\otimes\cdots\otimes W^+_m)/} \\ &\hspace{60pt} (U^+_1\otimes W^+_2\otimes\cdots\otimes W^+_m 
+\sum_{i=2}^mU_1\otimes W^+_2\otimes\cdots\otimes W_i\otimes\cdots\otimes W^+_m)
\end{align*}
which is an irreducible, tame, quasi-finite, lowest weight, $\Z^n$-graded $\E_n$-module. \qquad $\Box$
\end{prop}

The module $\Hc^{\bs k,(\bs x,\bs y,\bs z)}_{\bs\al,\bs\beta,\bs \ga}(\bs u)$ has level $q_3^{|\bs x|/2}q_1^{|\bs y|/2}q_2^{|\bs z|/2}$, where for an $m$-tuple $\bs s=(s_1,\dots,s_m)$ we set
$|\bs s|=s_1+\cdots+s_m$. It has a basis parametrized by $m$-tuples $(\La^{(1)},\ldots,\La^{(m)})$
of plane partitions with asymptotic conditions $\al^{(i)},\beta^{(i)},\ga^{(i)}$ 
such that $\La^{(i)}$ contains $(x_i,y_i,z_i)$ (if such a box can be in a plane partition with asymptotic conditions $\al^{(i)},\beta^{(i)},\ga^{(i)}$) 
and does not contain $(x_i+1,y_i+1,z_i+1)$
for $i=1,\ldots,m$.

As an example, consider the case $m=2,k_1=k,k_2=k+1, u_1=u,u_2=v$ and $(x_1,y_1,z_1)=(1,0,0)$, $(x_2,y_2,z_2)=(0,1,0)$. In other words, we consider the tensor product
 $\M^{(k)}(u;K_1)\otimes \M^{(k+1)}(v;K_2)$ such that $K_2^2=q_3u/v$ and $K_1^2=q_1v/u$. 
Then we have
\begin{align*}
\Hc^{(k,k+1),(1,0,0\,;\,0,1,0)}_{\emptyset,\emptyset,\emptyset}(u,v)=(\M^{(k)}(u;K_1))/
(U^{(k,k+1),(2,1,1)}(u,K_1))\otimes W^{(k+1,k),(1,2,1)}(v;K_2)
\end{align*}
is an irreducible subquotient. This module level $q^{-1}$, and has a
basis parameterized by pairs of partitions $(\la,\mu)$, where
$\la'_i=\La^{(1)}_{1,i}$, $\mu'_i=\La^{(2)}_{i,1}$.
Note that these partitions are colored differently from those of Fock spaces.

As another example, let $\Hc^{(k)}(u_1,\dots,u_n)$ correspond to the case $m=n$, $(x_i,y_i,z_i)=(0,1,0)$, $k_i=k+i-1$. This module has level $q_1^{n/2}$ and 
a basis parameterized by 
$n$ arbitrary partitions (colored differently from the Fock space). For special values of $u_i$ one should be able to find a subquotient of 
$\Hc^{(k)}(u_1,\dots,u_n)$ isomorphic to $\mathcal G_{\al,\ga}^{(k)}(u)$. And on the other hand, the module
$\Hc^{(k)}(u_1,\dots,u_n)$ can be obtained from $\mathcal G_{\al,\emptyset}^{(k)}(u)$ by taking the limit
$\al_i'-\al_{i+1}'\to \infty$.

In fact, all lowest weight modules in this paper can be written as $\Hc^{(\bs k),(\bs x,\bs y,\bs z)}_{\bs\al,\bs\beta,\bs \ga}(\bs u)$ for some values of parameters.
\begin{conj}
Every irreducible, quasi-finite, lowest weight $\E_n$-module is isomorphic to a module 
$\Hc^{(\bs k),(\bs x,\bs y,\bs z)}_{\bs\al,\bs\beta,\bs \ga}(\bs u)$ with suitable choice of parameters.
\end{conj}

\section{Structure of ${\mathcal G}_{\mu,\nu}^{(k)}(u)$}\label{Weyl type sec}

Consider the modules ${\mathcal G}_{\mu,\nu}^{(k)}(u)$ defined in Section \ref{special-vallue}
(we change the letters for partitions from $\al,\ga$ to $\mu,\nu$
so as not to mix them with simple roots).
They are parametrized by $k\in\hat{I}$ and a pair of partitions $\mu,\nu$ 
such that $\mu_1,\nu_1< n$ and $\nu$ is colorless.  
The aim of this section is to clarify the structure of ${\mathcal G}_{\mu,\nu}^{(k)}(u)$ 
as a module over  the horizontal subalgebra $U_q^{hor}\widehat{\mathfrak{gl}}_n$. 

\subsection{Colorless partitions of width $n$}\label{colorless and roots}

Here we collect further information on colorless partitions. 

Let $\{\al_i\}_{i\in \hat{I}}$ 
be the set of simple roots of $\widehat{\mathfrak{sl}}_n$. 
As before we retain the modulo $n$ convention for indexes.

The null root is denoted by $\delta=\al_1+\cdots+\al_n$ as usual.
The set of positive real roots consists of the following elements:

\begin{align*}
\beta^{(m)}_{i,j}=\al_i+\al_{i+1}+\cdots+\al_{j-1}+m\delta\quad
(1\leq i<j\leq n; m\geq0),\\
\beta^{(m)}_{i,j}=\al_i+\al_{i+1}+\cdots+\al_{j-1}+m\delta\quad
(1\leq j<i\leq n;m\geq0).
\end{align*}

A positive real root $\beta^{(m)}_{i,j}$ is called minimal if and only if $m=0$.
A positive real root $\beta$ is minimal if and only if $(\beta,\rho)<n$.
There are $n(n-1)$ minimal elements among the positive real roots.
For example, for $n=3$, the following are minimal:
\begin{align*}
\al_1,\al_2,\al_3,\al_1+\al_2,\al_2+\al_3,\al_3+\al_1.
\end{align*}

As in Section \ref{colorless}, let $c_i$ ($i\in I$) be generators of a free $\Z$-module
of rank $n-1$, and let $c_0=-\sum_{i=1}^{n-1}c_i$. For $i,j\in \Z$, we set $c_{i,j}=c_{i-1}+\cdots+c_{j'}$ 
where $j'\le i$ and $j'\equiv j$. 

\begin{lem}\label{lem-c}
Let $i_{m}$ ($m=1,\ldots,n$) be integers. Then the following are equivalent.
\begin{enumerate}
\item Modulo $n$ the set $\{i_1,\cdots,i_n\}$ coincides with $\{1,\ldots,n\}$, 
\item $\sum_{m=1}^nc_{i_{m}}=0$, 
\item $\sum_{m=1}^nc_{i_{m},{m}}=0$.
\end{enumerate}
\end{lem} 
\begin{proof}
The equivalence of (i) and (ii) is obvious. 
We have $c_{i,j}+c_{l,m}=c_{i,m}+c_{l,j}$ for all $i,j,l,m\in\Z$, hence (ii) implies (iii). 
Set $\bar{c}_i=c_{i,1}$. Then $c_{i,j}=\bar{c}_i-\bar{c}_j$, so that (iii) 
is equivalent to $\sum_{m=1}^n\bar{c}_{i_m}=\sum_{m=1}^n\bar{c}_{m}$. 
It is easy to see that this implies (ii). 
\end{proof}

In this subsection, we fix the color of $(1,1)\in \Z^2$ to be $n$. 
Recall that $\nu$ is said to be $n$-colorless (hereafter colorless, for short) if and only if 
$\sum_{(x,y)\in Y(\nu)}c_{x-y}=0$. Lemma \ref{lem-c} implies 
\begin{lem}\label{perm}
Suppose $\nu_1\le n$. Then $\nu$ is colorless if and only if 
modulo $n$ the set of integers
$\{\nu'_m-m\}_{m=1}^{n}$
coincides with $\{1,2,\ldots,n\}$.
\end{lem}
Introduce the set 
\begin{align*}
&\mathcal{P}^{(n)}_0:=\{\nu\mid \text{$\nu$ is a colorless partition such that $\nu_1\le n$, $\nu'_n<n$} \}.
\end{align*}
By Lemma \ref{perm}, with each $\nu\in \mathcal{P}^{(n)}_0$ is associated 
a permutation $\sigma_\nu\in \mathfrak{S}_n$ given by  
\begin{align*}
\sigma_{\nu}(n-m+1)=\overline{\nu'_{m}+1-m}\,,
\end{align*}
where $\overline{\nu'_{m}+1-m}\in\{1,2,\ldots,n\}$ and $\overline{\nu'_{m}+1-m}\equiv \nu'_{m}+1-m$. 
We call $\sigma_{\nu}(n-m+1)$ the color below the bottom of the $m$-th column 
of $\nu$.
We have 
\begin{align*}
\nu'_{n-\sigma^{-1}_\nu(m)+1}\equiv {m}-\sigma^{-1}_\nu(m)\quad\bmod n\,.
\end{align*}

For $\nu\in\mathcal P^{(n)}_0$ we define a set of positive real roots
$\beta_i(\nu)$ $(i=1,\ldots,n-1)$ by
\begin{align*}
\beta_i(\nu)=\al_{\nu'_{i+1}-i}+\al_{\nu'_{i+1}-i+1}+\cdots+\al_{\nu'_i-i}.
\end{align*}
For example, if $\nu=\emptyset$, we have $\beta_i(\emptyset)=\al_{n-i}$.
A partition $\nu\in\mathcal P^{(n)}_0$ is called minimal if and only if $\beta_i(\nu)$ are minimal.
This condition is equivalent to
\begin{align*}
\nu'_i<\nu'_{i+1}+n-1\text{ for all $i=1,\ldots,n-1$}.
\end{align*}
For each $\s\in\mathfrak S_n$ there exists a unique minimal partition $\nu$ such that $\s_\nu=\s$.
In other words, the number of minimal partitions is $n!$. We will show that they form a single orbit of
a certain action of the symmetric group $\mathfrak S_n$.\\

Let $\nu$ be a minimal partition in $\mathcal P^{(n)}_0$ such that
$m_1\buildrel{\rm def}\over=\s_\nu^{-1}(1)<m_2\buildrel{\rm def}\over=\s_\nu^{-1}(2)$.
Let $s_1=(12)\in\mathfrak S_n$ be the transposition.
We describe the minimal partition $\nu^{(1)}$ satisfying
\begin{align*}
\s_{\nu^{(1)}}=s_1\circ\s_\nu.
\end{align*}

Case 1:\quad $m_2=m_1+1$.

In this case we have $\nu'_{n-m_1+1}=\nu'_{n-m_1}$.
We define $\nu^{(1)}$ as follows:
\begin{align*}
{\nu^{(1)}}'_i=
\begin{cases}
\nu'_i&\text{ if $i>n-m_1+1$};\\
\nu'_{n-m_1+1}+1&\text{ if $i=n-m_1+1$};\\
\nu'_{n-m_1}+n-1&\text{ if $i=n-m_1$};\\
\nu'_i+n&\text{ if $i<n-m_1$}.
\end{cases}
\end{align*}
Note that compared to $\nu$ a box of color $1$ is added in $\nu^{(1)}$ on the 
$(n-m_1+1)$-th column, and
boxes of color $2,\ldots,n$ on the 
$(n-m_1)$-th column. Therefore $\nu^{(1)}$ is also colorless.
One can see that $\nu^{(1)}$ is minimal. In fact, suppose that
${\nu^{(1)}}'_{n-m_1+1}-{\nu^{(1)}}'_{n-m_1+2}\geq n-1$.
Since $\nu'_{n-m_1+1}-\nu'_{n-m_1+2}<n-1$, 
this means $\nu'_{n-m_1+1}-\nu'_{n-m_1+2}=n-2$, 
and therefore
$\s_\nu(m_1-1)=2$. This is a contradiction to $\s_\nu(m_1+1)=2$. Similarly,
${\nu^{(1)}}'_{n-m_1-1}-{\nu^{(1)}}'_{n-m_1}\geq n-1$ 
never happens. This implies $\nu^{(1)}$ is minimal.\\

Case 2:\quad $m_2>m_1+1$.

In this case, there exists $m$ such that $n-m_2+1<m<n-m_1+1$, and 
we have $\nu'_{n-m_2+1}>\nu'_{m}>\nu'_{n-m_1+1}$.
We define $\nu^{(1)}$ as follows:
\begin{align*}
{\nu^{(1)}}'_i=
\begin{cases}
\nu'_i&\text{ if $i>n-m_1+1$};\\
\nu'_{n-m_1+1}+1&\text{ if $i=n-m_1+1$};\\
\nu'_i&\text{ if $n-m_2+1<i<n-m_1+1$};\\
\nu'_{n-m_2+1}-1&\text{ if $i=n-m_2+1$};\\
\nu'_i&\text{ if $i<n-m_2+1$}.
\end{cases}
\end{align*}
One can easily show that $\nu^{(1)}$ is a minimal partition in $\mathcal P^{(n)}_0$.

Similarly, for $i=2,\ldots,n-1$, if
$m_i\buildrel{\rm def}\over=\s_\nu^{-1}(i)<m_{i+1}\buildrel{\rm def}\over=\s_\nu^{-1}(i+1)$,
one can describe the minimal partition $\nu^{(i)}$ in $\mathcal P^{(n)}_0$ such that
\begin{align*}
\s_{\nu^{(i)}}=s_i\circ\s_\nu.
\end{align*}
Set
\begin{align*}
r_{\al_i}(\beta)=\beta-(\beta,\al_i)\al_i.
\end{align*}
\begin{prop}\label{ribeta}
For $i,j=1,\ldots,n-1$ we have
\begin{align*}
\beta_j(\nu^{(i)})=
r_{\al_i}(\beta_j(\nu)),\text{ or }
r_{\al_i}(\beta_j(\nu))+\delta.
\end{align*}
\end{prop}
\begin{proof} 
For simplicity we give a proof for $i=1$. The general case is similar.

In Case 1, if $j\not=
n-m_1-1,n-m_1,n-m_1+1$, we have $\beta_j(\nu)=\beta_j(\nu^{(1)})$.
One can easily see that $r_{\al_1}(\beta_j(\nu))=\beta_j(\nu)$
because neither color $1$ nor $2$ can appear below the bottom of the $j$-th
and $(j+1)$-st column of $\nu$.

For $j=n-m_1+1$, we have
\begin{align*}
\beta_j(\nu)=\al_m+\cdots+\al_n=\beta_j(\nu^{(1)})-\al_1,
\end{align*}
where $m+\nu'_j-\nu'_{j+1}=n$ and $2<m\leq n$.
Therefore, we have 
$\beta_j(\nu^{(1)})=r_{\al_1}(\beta_j(\nu))$.

For $j=n-m_1$, we have
\begin{align*}
&\beta_j(\nu)=\al_1,\\
&\beta_j(\nu^{(1)})=\al_2+\cdots+\al_n=r_{\al_1}(\beta_j(\nu))+\delta.
\end{align*}

For $j=n-m_1+1$, we have
\begin{align*}
&\beta_j(\nu)=\al_2+\cdots+\al_m=\beta_j(\nu^{(1)})-\al_1,
\end{align*}
where $m=2+\nu'_j-\nu'_{j+1}$ and $2\leq m<n$.
Therefore, we have 
$\beta_j(\nu^{(1)})=r_{\al_1}(\beta_j(\nu))$.

In Case 2, it is enough to consider the case $j=n-m_1+1,n-m_1,n-m_2+1,n-m_2$.

For $j=n-m_1+1$ we have
\begin{align*}
\beta_j(\nu)=\al_m+\cdots+\al_n=\beta_j(\nu^{(1)})-\al_1,
\end{align*}
where $m+\nu'_j-\nu'_{j+1}=n$ and $2<m\leq n$.

For $j=n-m_1$ we have
\begin{align*}
\beta_j(\nu)=\beta_j(\nu^{(1)})+\al_1=\al_1+\al_2+\cdots+\al_m,
\end{align*}
where $m=\nu'_j-\nu'_{j+1}+1<n$ and $2\leq m<n$.

For $j=n-m_2+1$
\begin{align*}
\beta_j(\nu)=\beta_j(\nu^{(1)})+\al_1
=\al_m+\cdots+\al_n+\al_1
\end{align*}
where $m+\nu'_j-\nu'_{j+1}=n+1$ and $2< m\leq n$.

For $j=n-m_2$
\begin{align*}
\beta_j(\nu)=\beta_j(\nu^{(1)})-\al_1=\al_2+\cdots+\al_m,
\end{align*}
where $m=2+\nu'_j-\nu'_{j+1}$ and $2\leq m<n$.

In all cases we have $\beta_j(\nu^{(1)})=r_{\al_1}(\beta_j(\nu))$.
\end{proof}

Let $Q=\oplus_{i=0}^{n-1}\Z\al_i$ be the root lattice of $\widehat{\mathfrak{sl}}_n$.
We define the action of symmetric group $\mathfrak S_n$ on $Q$ by letting 
$s_i=(i,i+1)$ act as $r_{\al_i}$ for $i\in I$.
Let $Q_{\geq0}$ be the subset of $Q$  consisting of positive real roots. 
We define a mapping
\begin{align*}
&\beta:\mathcal P^{(n)}_0\rightarrow Q_{\geq0}^{n-1},
\quad \nu\mapsto \beta(\nu)
=(\beta_1(\nu),\ldots,\beta_{n-1}(\nu)).
\end{align*}
This map is injective.
We have
$\beta(\emptyset)=(\al_{n-1},\ldots,\al_{1})$.
We denote by $\mathcal R^{(n)}_{\rm min}$ the $\mathfrak S_n$ orbit of this element.
Let us denote by $\mathcal P^{(n)}_{0,{\rm min}}$
the subset of $\mathcal P^{(n)}_0$ consisting of minimal partitions. 
Then we have 
$\beta(\mathcal P^{(n)}_{0,{\rm min}})=\mathcal R^{(n)}_{\rm min}$.
Set
\begin{align*}
\mathcal R^{(n)}=\{(\beta_1+m_1\delta,\ldots,\beta_{n-1}+m_{n-1}\delta)\in Q_{\geq0}^{n-1}
\ |\ (\beta_1,\ldots,\beta_{n-1})\in\mathcal R^{(n)}_{\rm min},
m_1,\ldots, m_{n-1}\geq 0\}.
\end{align*}

\begin{prop}
The following map is a bijection:
\begin{align*}
\beta
:\mathcal P^{(n)}_0\rightarrow\mathcal R^{(n)}\,.
\end{align*}
\end{prop}

\medskip

\noindent{\it Remark.}\quad 
So far we have fixed the color $n$ of $(1,1)\in\Z^2$. 
Removing rows of length $n$ from $\nu\in \mathcal{P}^{(n)}_0$ 
one obtains colorless partitions of width $<n$, whose top left corner carries 
a color different from $n$ in general. 
Obviously $\mathcal{P}^{(n)}_0$ is in bijective correspondence with 
$\hat{I}\times \bar{\mathcal{P}}^{(n)}_0$, where
\begin{align*}
\bar{\mathcal{P}}^{(n)}_0:=\{\nu\mid \text{$\nu$ is a colorless partition with $\nu_1<n$}\}.
\end{align*}
In the next section we shall fix $k\in\hat{I}$ and consider 
colorless partitions in $\bar{\mathcal{P}}^{(n)}_0$. 
\subsection{Weyl type modules}\label{G sec}
Set $\mathfrak{h}^*=\oplus_{i=0}^{n-1}\C\omega_i\oplus\C\delta$. 
Note that $(\omega_i,\al_j)=\delta_{i,j}.$ 
For $\lambda\in\mathfrak{h}^*$, we denote by $\mathcal{L}_{-\lambda}$ 
the irreducible lowest weight $U_q^{hor}\widehat{\mathfrak{gl}}_n$-module with lowest weight $-\lambda$. 
In the following, for a lowest weight module $M$,  
we shall consider the principal $\Z$-gradation $M=\oplus_{j\in\Z_{\ge0}}M_j$. 
We denote the principal character by 
\begin{align*}
\chi(M)=\sum_{j\in\Z_{\ge0}}x^j\dim M_j\,.
\end{align*}

Consider the following conditions (i),(ii) for $\lambda\in\mathfrak{h}^*$:
\begin{enumerate}
\item There exist positive real roots $\beta_1,\ldots,\beta_{n-1}$ such that 
\begin{align*}
&(\beta_i,\beta_j)=a_{i,j},\quad (\lambda+\rho,\beta_i)\in \Z_{>0}\,,\quad i,j\in I.
\end{align*}
\item $\lambda$ has an irrational level, $(\lambda+\rho,\delta)\not\in\Q$.
\end{enumerate}
The corresponding module $\mathcal{L}_{-\lambda}$ is a Weyl type module 
in the sense mentioned in Introduction, and its character is given as follows. 
\begin{prop}\label{KT}
Let $\lambda$, $\beta_i$ be as above, and denote by 
$W(\lambda)$ the subgroup of
the Weyl group of $\widehat{\mathfrak{sl}}_n$
generated by the reflections $r_{\beta_i}$, $i\in I$. 
Then the principal character of $\mathcal{L}_{-\lambda}$ is given by 
\begin{align}
\chi(\mathcal{L}_{-\lambda})=\frac{1}{(x)^n_\infty}\sum_{w\in W(\lambda)}(-1)^{\ell(w)}
x^{(\lambda+\rho-w(\lambda+\rho),\rho)}\,,
\label{chiL}
\end{align}
where $(z)_\infty=\prod_{j=1}^\infty(1-z^j)$.
\end{prop}
\begin{proof}
Under the conditions above, we have an isomorphism 
\begin{align}
\mathfrak{S}_n\overset{\sim}{\longrightarrow} W(\lambda)\,,
\quad r_{\alpha_i}\mapsto r_{\beta_i}\,,
\label{Weyl-gp}
\end{align}
where $\mathfrak{S}_n$ is
the Weyl group of $\mathfrak{sl}_n$
generated by the simple reflections $r_{\alpha_i}$, $i\in I$.
Let $M(-\lambda)$ (resp. $L(-\lambda)$) be the Verma module (resp. irreducible module) 
of $\widehat{\mathfrak{sl}}_n$ with lowest weight $-\lambda$, 
and let $\mathrm{ch}\,M(-\lambda)$ (resp. $\mathrm{ch}\,L(-\lambda)$) be its character. 
In this setting, a result of Kashiwara and Tanisaki (\cite{KT}, Theorem 1.1) states that 
\begin{align*}
\mathrm{ch}\,L(-\lambda)=\sum_{w\in W(\lambda)}(-1)^{\ell(w)}\mathrm{ch}\,M\bigl(-w(\lambda+\rho)+\rho\bigr)\,.
\end{align*}
Passing to the principal character we find
\begin{align*}
\chi\bigl(L(-\lambda)\bigr)=\frac{(x^n)_\infty}{(x)_\infty^n}
\sum_{w\in W(\lambda)}(-1)^{\ell(w)}x^{(\lambda+\rho-w(\lambda+\rho),\rho)}\,,
\end{align*}
where we have assigned degree $0$ to lowest weight vectors of $M(-\lambda)$ and $L(-\lambda)$,  
and used that the principal character of $M(-\lambda)$ is 
$(x^n)_\infty/(x)_\infty^n$. 
Multiplying both sides by the principal character $1/(x^n)_\infty$ of the horizontal Heisenberg algebra, 
we obtain the result. 
\end{proof}

Now fix $k\in \hat{I}$. 
In the rest of this section, we assign color $k$ to $(1,1)\in \Z^2$, though in Section  \ref{colorless and roots} we assigned color $0$ to it.
Let $\mu$, $\nu$ be partitions with $\mu_1,\nu_1<n$ and $\nu$ colorless. 

The module ${\mathcal G}_{\mu,\nu}^{(k)}(u)$ is a quotient of 
the tensor product $\mathcal{F}\otimes\mathcal{M}$, where 
\begin{align*}
\mathcal{F}=
\bigl(\mathcal{F}^{(k-n+1)}\bigr)^{\otimes \mu'_{n-1}}\otimes 
\bigl(\mathcal{F}^{(k-n+2)}\bigr)^{\otimes (\mu'_{n-2}-\mu'_{n-1})}\otimes 
\cdots \bigl(\mathcal{F}^{(k-1)}\bigr)^{\otimes(\mu'_1-\mu'_2)} 
\end{align*}
and $\mathcal{M}
=\mathcal{M}^{(k)}_{\emptyset,\emptyset,\nu}(u_{\mu'_1+1};q_1^{n/2}q^{-\mu'_1})$.

(We have suppressed the  parameters $u_j$ which are irrelevant for the following calculation.)

For $l=1,\ldots,n$, let
$\nu_{\ge l}$ denote the partition conjugate to $(\nu'_{l},\cdots,\nu'_n)$ (note that $\nu'_n=0$).
We are to calculate the eigenvalue of  $K_{\beta_i}$
on the vector $v'\otimes v''$,
where $v'\in \mathcal{F}$ is the vector
\begin{align}
v'=
\bigl(|\emptyset\rangle^{(k-n+1)}\bigr)^{\otimes \mu'_{n-1}}
\otimes 
\bigl(|\nu_{\ge n-1}\rangle^{(k-n+2)}\bigr)^{\otimes (\mu'_{n-2}-\mu'_{n-1})}
\otimes 
\cdots 
\otimes 
\bigl(|\nu_{\ge 2}\rangle^{(k-1)}\bigr)^{\otimes(\mu'_1-\mu'_2)},
\label{v'}
\end{align}
and $v''=v^{(k)}_{\emptyset,\emptyset,\nu}\in\mathcal{M}$.\\

We denote by $\sharp(CC_j)$ (resp., $\sharp(CV_j)$) the total number of color $j$ concave
(resp., convex) corners in one of those colored partitions which appear in \eqref{v'}.
We prepare a combinatorial lemma which can be easily proved by case checking, cf. Lemma \ref{perm}.
\begin{lem}
Let $\tau\in\mathfrak S_n$ be an element uniquely determined by the condition
\begin{align*}
\tau(\overline{\nu'_j-j+k})=j\quad(1\leq j\leq n),
\end{align*}
where, as before, $\overline{\nu'_j-j+k}\in\{1,\ldots,n\}$ and $\overline{\nu'_j-j+k}\equiv\nu'_j-j+k$.
Then we have
\begin{align*}
\sharp(CC_j)-\sharp(CV_j)=\mu'_{\tau(j)}-\mu'_{\tau(j-1)}+\mu'_1\delta_{j,k}.
\end{align*}
\end{lem}

From now on, for simplicity, we abbreviate the overline notation and use $j$ for $\overline j$.

Let us determine the lowest weight of $\mathcal{G}^{(k)}_{\mu,\nu}$ 
with respect to $U_q^{hor}\widehat{\mathfrak{sl}}_n$. 
For that purpose, we define from $\nu$ a set of positive real roots $\beta^{(k)}_i(\nu)$ by 
\begin{align*}
&\beta^{(k)}_i(\nu)=\alpha_{\nu'_{i+1}+k-i}+\alpha_{\nu'_{i+1}+k-i+1}+\cdots+\alpha_{\nu'_{i}+k-i}\,,\quad
i\in I. 
\end{align*}
From Proposition \ref{ribeta}, we have 
\begin{align*}
&(\beta_i^{(k)}(\nu),\beta^{(k)}_j(\nu))=a_{i,j}\quad (i,j\in I),
\\
&(\rho,\beta_i^{(k)}(\nu))=\nu'_i-\nu'_{i+1}+1 \,.
\end{align*}

Define an $\widehat{\mathfrak{sl}}_n$-weight $\lambda^{(k)}(\mu,\nu)\in \mathfrak{h}^*$ by
\begin{align}
&(\lambda^{(k)}(\mu,\nu)+\rho,\beta_i^{(k)}(\nu))=\mu'_{i}-\mu'_{i+1}+1\quad (i\in I), \label{la weight1}
\\
&(\lambda^{(k)}(\mu,\nu)+\rho,\delta)=\left(\frac{t}{2}+1\right)n\,,\label{la weight2}
\end{align}
where we have set $q_1=q^{t}$. 
Since $q_1^lq_2^m\neq 1$ for $(l,m)\in\Z^2\backslash\{(0,0)\}$, we have $t\not\in\Q$. 

For $\beta=\sum_{i=0}^{n-1}m_i\alpha_i$, we write $K_\beta=\prod_{i=0}^{n-1}K_i^{m_i}$

\begin{prop}\label{hwt-G}
Let  $v^{(k)}_{\mu,\nu}$ be the lowest weight vector
of $\mathcal G_{\mu,\nu}^{(k)}(u)$. Then we have
\begin{align*}
&K_{\beta_i^{(k)}(\nu)}v^{(k)}_{\mu,\nu}=q^{-(\lambda^{(k)}(\mu,\nu),
\beta_i^{(k)}(\nu))}v^{(k)}_{\mu,\nu} \quad (i\in I),\\
&
K_\delta\,v^{(k)}_{\mu,\nu}=
q^{-(\lambda^{(k)}(\mu,\nu),\delta)}v^{(k)}_{\mu,\nu}\,.
\end{align*}
\end{prop}
\begin{proof}
It is sufficient to show the first equation
since the second represents the level $K=q^{(\lambda^{(k)}(\mu,\nu),\delta)}$ of representation.

Since $\nu\in\bar{\mathcal{P}}_0$, for each color $j\in\{1,\ldots,n\}$ there is a unique column
$\tau(j)\in\{1,\ldots,n\}$ such that $\nu'_{\tau(j)}+k-\tau(j)\equiv j$
(see Lemma \ref{perm}).
In general, for a vector $|\lambda\rangle^{(l)}$ in $\mathcal{F}^{(l)}(u)$ the eigenvalue of 
$K_j$ is $q^m$, where $m$ is the number of convex corners of color $j$ minus the number of 
concave corners of the same color(see Lemma \ref{K act 2}). Therefore, we have
\begin{align*}
K_iv'=q^{-\mu'_{\tau(i)}+\mu'_{\tau(i-1)}-\mu'_1\delta_{i,k}}v'.
\end{align*}
Set
\begin{align*}
\sharp_i(k,\nu)=\sum_{j=k-i+\nu'_{i+1}}^{k-i+\nu'_i}\delta^{(n)}_{j,k}.
\end{align*}
Noting that $\tau(\nu'_j-j+k)=j$ for $j=1,\ldots,n$, we obtain
\begin{align*}
K_{\beta^{(k)}_i(\nu)}v'=q^{\mu'_{i+1}-\mu'_i-\sharp_i(k,\nu)}v'\,.
\end{align*}
Similarly, from Theorem \ref{M2 thm} we calculate 
\begin{align*}
K_{\beta^{(k)}_i(\nu)}v''=q^{\nu'_{i}-\nu'_{i+1}+\sharp_i(k,\nu)}v''\,.
\end{align*}

Hence the eigenvalue of $K_{\beta^{(k)}_i(\nu)}$ on $v'\otimes v''$ is  
\begin{align*}
q^{\mu'_{i+1}-\mu'_i}q^{\nu'_i-\nu'_{i+1}}=q^{-(\lambda^{(k)}(\mu,\nu),\beta^{(k)}_i(\nu))}\,.
\end{align*}
\end{proof}

\begin{prop}\label{char-G}
The character of ${\mathcal G}_{\mu,\nu}^{(k)}(u)$ is given by  
\begin{align*}
\chi\bigl({\mathcal G}_{\mu,\nu}^{(k)}(u)\bigr)=
\chi\bigl(\mathcal{L}_{-\lambda^{(k)}(\mu,\nu)}\bigr)\,.
\end{align*}
\end{prop}
\begin{proof}
Introduce the notation
\begin{align*}
\xi=\sum_{i=1}^{n-1}(\nu'_i-\nu'_{i+1})(\omega_i-\omega_0)\,,
\quad
\eta=\sum_{i=1}^{n-1}(\mu'_i-\mu'_{i+1})(\omega_i-\omega_0)\,.  
\end{align*}
By construction, the module ${\mathcal G}_{\mu,\nu}^{(k)}(u)$ 
has a basis indexed by $n$-tuples of partitions
$(\lambda^{(1)},\ldots,\lambda^{(n)})$ such that 
\begin{align}
\lambda^{(j)}_{i}\ge \lambda^{(j+1)}_{i+\nu'_j-\nu'_{j+1}}-\mu'_j+\mu'_{j+1}
\quad \text{for $j\in I$, $i\in \Z_{>0}$.}  
\label{Pab}
\end{align}
Hence the principal character of ${\mathcal G}_{\mu,\nu}^{(k)}(u)$ coincides with the generating function 
\begin{align*}
\chi_{\xi,\eta}=
\sum_{(\lambda^{(1)},\ldots,\lambda^{(n)})\in P_{\xi,\eta}}  
x^{\sum_{j=1}^n\sum_{i=1}^\infty\lambda^{(j)}_i}\,,
\end{align*}
where $P_{\xi,\eta}$ denotes
the set of $(\lambda^{(1)},\ldots,\lambda^{(n)})$ satisfying \eqref{Pab}. 
In \cite{FFJMM2}, Theorem 4.6, we have obtained the formula
\begin{align}
\chi_{\xi,\eta}=
\frac{1}{(x)^n_\infty}\sum_{s\in \mathfrak{S}_n}(-1)^{\ell(s)}
x^{(\xi+\rho-s(\xi+\rho),\eta+\rho)}\,.
\label{chixieta}
\end{align}

Let us compare this formula with \eqref{chiL}.
For $s\in \mathfrak{S}_n$, let $w_s\in W(\la)$ be the image of $s$ under the isomorphism \eqref{Weyl-gp}. 
With the aid of the relations 
\begin{align*}
&(\xi+\rho,\alpha_i)=\nu'_i-\nu'_{i+1}+1=(\rho,\beta_i^{(k)}(\nu))\,,
\\
&(\eta+\rho,\alpha_i)=\mu'_i-\mu'_{i+1}+1=(\lambda+\rho,\beta_i^{(k)}(\nu))\,,
\end{align*}
it is easy to see that 
\begin{align*}
(\xi+\rho-s(\xi+\rho),\eta+\rho)=(\lambda+\rho-w_s(\lambda+\rho),\rho)\quad
(\forall s\in \mathfrak{S}_n).
\end{align*}
Therefore, in the sum \eqref{chixieta} and \eqref{chiL}, the corresponding terms coincide. 
The proof is over.
\end{proof}

Summing up we arrive at the conclusion.
\begin{thm}
We have an isomorphism of $U_q^{hor}\widehat{\mathfrak{gl}}_n$ modules 
\begin{align*}
{\mathcal G}_{\mu,\nu}^{(k)}(u)\simeq \mathcal{L}_{-\lambda^{(k)}(\mu,\nu)},
\end{align*}
where $\lambda^{(k)}(\mu,\nu)$ is defined by \Ref{la weight1},\eqref{la weight2}.
 \end{thm}
\begin{proof}
Proposition \ref{hwt-G} tells that both sides have the same lowest weight. 
By Proposition \ref{char-G}, their characters also coincide. Since 
$\mathcal{L}_{-\lambda^{(k)}(\mu,\nu)}$ is irreducible, the two modules must be isomorphic. 
\end{proof}

\medskip

{\bf Acknowledgments.}
Research of MJ is supported by the Grant-in-Aid for Scientific Research B-23340039.
Research of TM is supported by the Grant-in-Aid for Scientific Research B-22340031.
Research of EM is supported by NSF grant DMS-0900984. 
The present work has been carried out during the visits of BF and EM 
to Kyoto University. They wish to thank the University for hospitality. 

We would like to thank M. Kashiwara for useful discussions.


\begin{thebibliography}{00000}
\bibitem[FFJMM1]{FFJMM1}
B. Feigin, E. Feigin, 
M. Jimbo, T. Miwa and E. Mukhin,
{\it Quantum continuous $\mathfrak{gl}_\infty$:
Semi-infinite construction of representations},
Kyoto J. Math. {\bf 51} (2011), no. 2, 337--364

\bibitem[FFJMM2]{FFJMM2}
B. Feigin, E. Feigin, M. Jimbo, T. Miwa and E. Mukhin,
{\it Quantum continuous $\mathfrak{gl}_\infty$:
Tensor product of Fock modules and $\mathcal{W}_n$ characters},
Kyoto J. Math. {\bf 51} (2011), no. 2, 365--392 

\bibitem[FFNR]{FFNR} B. Feigin, M. Finkelberg, A. Negut, and L. Rybnikov,
{\it Yangians and cohomology rings of Laumon spaces}, Selecta Math. (N.S.) {\bf 17} (2011), no.
3, 573--607


\bibitem[FJMM]{FJMM} B. Feigin, M. Jimbo, T. Miwa, and E. Mukhin, {\it Quantum toroidal $\mathfrak{gl}_1$ algebra : plane partitions}, arXiv:1110.5310, 1--38

\bibitem[GKV]{GKV} V. Ginzburg, M. Kapranov, and E. Vasserot,
{\it Langlands reciprocity for algebraic surfaces}, Math. Res. Lett. {\bf 2} (1995), no. 2, 147--160

\bibitem[H]{H}  D. Hernandez, {\it Quantum toroidal algebras and their representations}, Selecta Math. (N.S.) {\bf 14} (2009), no. 3-4, 701--725

\bibitem[KT]{KT} M. Kashiwara and T. Tanisaki, {\it Characters of
irreducible modules with non-critical highest weights over affine Lie
algebras}, Representations and quantizations (Shanghai, 1998), China High.
Educ. Press, Beijing, (2000), 275--296

\bibitem[M]{M} K. Miki, {\it A $(q,\gamma)$ analog of the $W_{1+\infty}$ algebra}, Journal of Math. Phys. {\bf 48} (2007), no. 12, 123520 , 1--35

\bibitem[M2]{M99} K. Miki, 
{\it Toroidal braid group action and an automorphism of toroidal algebra $U_q\bigl(\mathfrak{sl}_{n+1,tor}\bigr)$}, 
Lett. Math. Phys. {\bf 47} (1999), 365--378.

\bibitem[N]{N} H. Nakajima, {\it Quiver varieties and quantum affine algebras}, Translation
of Sugaku {\bf 52} (2000), no. 4, 337--359; Sugaku Expositions {\bf 19} (2006), no. 1, 53--78

\bibitem[TU]{TU} K. Takemura and D. Uglov, {\it Representations of the quantum toroidal algebra on highest weight modules of the quantum affine algebra of type $\gl_N$}, Publ. Res. Inst. Math. Sci. {\bf 35} (1999), no. 3, 407--450

\bibitem[VV]{VV} M. Varagnolo and E. Vasserot, {\it Schur duality in the toroidal setting}, Comm. Math. Phys. {\bf 182} (1996), no. 2, 469--483
\end{thebibliography}
\end{document}